\documentclass[11pt,reqno]{amsart}
\usepackage{amssymb,latexsym,hyperref,geometry}
\hypersetup{backref,linkbordercolor={.9 .9 .9},citebordercolor={.9 .9 .9}}
\title[Central limit theorems for additive functionals of ergodic
diffusions]{Central limit theorems for additive functionals of ergodic Markov
  diffusions processes} %
\date{Preprint April, 2011. Compiled \today} %
\author{Patrick Cattiaux} %
\address[P. Cattiaux]{UMR CNRS 5219, Universit\'e de Toulouse, France} %
\email{cattiaux(at)math.univ-toulouse.fr} %
\author{Djalil Chafa{\"{\i}}} %
\address[D. Chafa{\"{\i}}]{UMR CNRS 8050,
  Universit\'e Paris-Est Marne-la-Vall\'ee, France} %
\email{djalil(at)chafai.net} %
\urladdr{\url{http://djalil.chafai.net/}} %
\author{Arnaud Guillin} %
\address[A. Guillin]{UMR CNRS 6620, Universit\'e Blaise Pascal et Institut Universitaire de France, France} %
\email{guillin(at)math.univ-bpclermont.fr} %
\urladdr{\url{http://math.univ-bpclermont.fr/~guillin/}} %
\dedicatory{Dedicated to the Memory of Naoufel Ben Abdallah}

\numberwithin{equation}{section}
\newtheorem{theorem}[equation]{Theorem}
\newtheorem{proposition}[equation]{Proposition}
\newtheorem{lemma}[equation]{Lemma}
\newtheorem{corollary}[equation]{Corollary}
\newtheorem{definition}[equation]{Definition}
\newtheorem{assumption}[equation]{Assumption}
\newtheorem{remark}[equation]{Remark}

\newcommand{\De}{\Delta}
\newcommand{\Ga}{\Gamma}
\newcommand{\La}{\Lambda}
\newcommand{\al}{\alpha}
\newcommand{\be}{\beta}
\newcommand{\de}{\delta}
\newcommand{\ep}{\varepsilon}
\newcommand{\ga}{\gamma}
\newcommand{\la}{\lambda}
\newcommand{\na}{\nabla}
\newcommand{\sig}{\sigma}
\newcommand{\vphi}{\varphi}
\newcommand{\Var}{\mathrm{Var}}
\newcommand{\Osc}{\mathrm{Osc}}

\newcommand{\dD}{\mathbb{D}}
\newcommand{\dE}{\mathbb{E}}

\newcommand{\dL}{\mathbb{L}}

\newcommand{\dP}{\mathbb{P}}
\newcommand{\dR}{\mathbb{R}}

\newcommand{\cD}{\mathcal{D}}
\newcommand{\cE}{\mathcal{E}}\newcommand{\cF}{\mathcal{F}}
\newcommand{\cG}{\mathcal{G}}

\newcommand{\cL}{\mathcal{L}}
\newcommand{\cN}{\mathcal{N}}

\newcommand{\ABS}[1]{{{\left| #1 \right|}}} 
\newcommand{\DP}[1]{{{\left<#1\right>}}} 
\newcommand{\NRM}[1]{{{\left\| #1\right\|}}} 
\newcommand{\PAR}[1]{{{\left(#1\right)}}} 
\newcommand{\pd}{{\partial}} 
\newcommand{\SBRA}[1]{{{\left[#1\right]}}} 

\begin{document}

\begin{abstract}
  We revisit functional central limit theorems for additive functionals of
  ergodic Markov diffusion processes. Translated in the language of partial
  differential equations of evolution, they appear as diffusion limits in the
  asymptotic analysis of Fokker-Planck type equations. We focus on the square
  integrable framework, and we provide tractable conditions on the
  infinitesimal generator, including degenerate or anomalously slow
  diffusions. We take advantage on recent developments in the study of the
  trend to the equilibrium of ergodic diffusions. We discuss examples and
  formulate open problems.
\end{abstract}

\maketitle

{\footnotesize\tableofcontents}

\section{Introduction}
\label{Intro}

Let ${(X_t)}_{t\geq0}$ be a continuous time strong Markov process with state
space $\dR^d$, non explosive, irreducible, positive recurrent, with unique
invariant probability measure $\mu$. Following \cite[Theorem 5.1 page
170]{MR0112175}, 
for every $f\in\dL^1(\mu)$, if almost surely (a.s.) the function
$s\in\dR_+\mapsto f(X_s)$ is locally Lebesgue integrable, then
\begin{equation}\label{eq:SLLN}
  \frac{S_t}{t} %
  \underset{t\to\infty}{\overset{\text{a.s.}}{\longrightarrow}} %
  \int\!f\,d\mu %
  \quad\text{where}\quad %
  S_t:=\int_0^t\!f(X_s)\,ds.
\end{equation}
If $X_0\sim\mu$ then by the Fubini theorem \eqref{eq:SLLN} holds for all
$f\in\dL^1(\mu)$ and the convergence holds additionally in $\dL^1$ thanks to
the dominated convergence theorem. The statement \eqref{eq:SLLN} which relates
an average in time with an average in space is an instance of the ergodic
phenomenon. It can be seen as a strong law of large numbers for the additive
functional ${(S_t)}_{t\geq0}$ of the Markov process ${(X_t)}_{t\geq0}$. The
asymptotic fluctuations are described by a central limit theorem which is the
subject of this work. Let us assume that $X_0\sim\mu$ and $f\in\dL^2(\mu)$
with $\int\!f\,d\mu=0$ and $f\neq0$. Then for all $t\geq0$ we have
$S_t\in\dL^2(\mu)\subset\dL^1(\mu)$ and $\dE(S_t)=0$. We say that
${(S_t)}_{t\geq0}$ satisfies to a \emph{central limit theorem} (CLT) when
\begin{equation}\label{eq:CLT}\tag{CLT}
  \frac{S_t}{s_t} %
  \overset{\text{law}}{\underset{t\to\infty}{\longrightarrow}} %
  \cN(0,1)
\end{equation}
for a deterministic positive function $t\mapsto s_t$ which may depend on $f$.
Here $\cN(0,1)$ stands for the standard Gaussian law on $\dR$ with mean $0$
and variance $1$. By analogy with the CLT for i.i.d.\ sequences one may expect
that $s_t^2=\Var(S_t)$ and that this variance is of order $t$ as $t\to\infty$.
A standard strategy for proving \eqref{eq:CLT} consists in representing
${(S_t)}_{t\geq0}$ as a sum of an $\dL^2$-martingale plus a remainder term
which vanishes in the limit, reducing the proof to a central limit theorem for
martingales. This strategy is particularly simple under mild assumptions
\cite[VII.3 p. 486]{JS}. Namely, if $L$ is the infinitesimal generator of
${(X_t)}_{t\geq0}$ with domain $\dD(L)\subset\dL^2(\mu)$ and if $g\in\dD(L)$
then ${(M_t)}_{t\geq0}$ defined by
\[
M_t:=g(X_t)-g(X_0)-\int_0^t\!(Lg)(X_s)\,ds
\]
is a local $\dL^2$ martingale. Now if $g^2\in\dD(L)$ and
$\Gamma(g):=L(g^2)-2gLg\in\dL^1(\mu)$, then
\[
\DP{M}_t=\int_0^t\!\Gamma(g)(X_s)\,ds.
\]
The law of large numbers \eqref{eq:SLLN} yields
$\lim_{t\to\infty}t^{-1}\DP{M}_t=\int\!\Gamma(g)\,d\mu$. As a consequence, for
a prescribed $f$, if the Poisson equation $Lg=f$ admits a mild enough solution
$g$ then 
\[
\frac{M_t}{s_t}=\frac{g(X_t)-g(X_0)}{s_t}-\frac{S_t}{s_t}.
\]
This suggests to deduce \eqref{eq:CLT} from a CLT for martingales. We will
revisit this strategy. Beyond \eqref{eq:CLT}, we say that ${(S_t)}_{t\geq0}$
satisfies to a \emph{Functional Central Limit Theorem} (FCLT) or
\emph{Invariance Principle} when for every finite sequence $0<t_1\leq\cdots\leq
t_n<\infty$,
\begin{equation}\label{eq:FCLT}\tag{FCLT}
  \PAR{\frac{S_{t_1/\ep}}{s_{t_1/\ep}},%
    \ldots,\frac{S_{t_n/\ep}}{s_{t_n/\ep}}}
  \overset{\text{law}}{\underset{\ep\to0}{\longrightarrow}} %
  \cL\PAR{\PAR{B_{t_1},\ldots,B_{t_n}}}
\end{equation}
where ${(B_t)}_{t\geq0}$ is a standard Brownian Motion on $\dR$. Taking $n=1$
gives \eqref{eq:CLT}. To capture multitime correlations, one may upgrade the
convergence in law in \eqref{eq:FCLT} to an $\dL^2$ convergence. The statement
\eqref{eq:FCLT} means that as $\ep\to0$, the rescaled process
${(S_{t/\ep}/s_{t/\ep})}_{t\geq0}$ converges in law to a Brownian
Motion, for the topology of finite dimensional marginal laws. At the level of
Chapman-Kolmogorov-Fokker-Planck equations, \eqref{eq:FCLT} is a
\emph{diffusion limit} for a weak topology.

In this work, we focus on the case where ${(X_t)}_{t\geq0}$ is a Markov
\emph{diffusion} process on $E=\dR^d$, and we seek for conditions on $f$ and
on the infinitesimal generator in order to get \eqref{eq:CLT} or even
\eqref{eq:FCLT}. We shall revisit the renowned result of Kipnis and Varadhan
\cite{KV}, and provide an alternative approach which is not based on the
resolvent. Our results cover fully degenerate situations such as the kinetic
model studied in \cite{gautetal,DegSeb,CCM}. More generally, we believe that a
whole category of diffusion limits which appear in the asymptotic analysis of
evolution partial differential equations of Fokker-Planck type enters indeed
the framework of the central limit theorems we shall discuss. We also explain
how the behavior \emph{out of equilibrium} (i.e.\ $X_0\not\sim\mu$) may be
recovered from the behavior \emph{at equilibrium} (i.e.\ $X_0\sim\mu$) by
using propagation of chaos (decorrelation), for instance via Lyapunov criteria
ensuring a quick convergence in law of $X_t$ to $\mu$ as $t\to\infty$. Note
that since we focus on an $\dL^2$ framework, the natural normalization is the
square root of the variance and we can only expect Gaussian fluctuations. We
believe however that stable limits that are not Gaussian, also known as
``anomalous diffusion limits'', can be studied using similar tools (one may
take a look at the works \cite{JKO,MMM} in this direction).

The literature on central limit theorems for discrete or continuous Markov
processes is immense and possesses many connected components. Some instructive
entry points for ergodic Markov processes are given by
\cite{DL1,DL2,DL3,Cunylin,HP,KM03,MR2144185,KM05,GM,PV1,PV2,PV3,MR2045987}. We
refer to \cite{KLO} and \cite{HL} for null recurrent Markov processes. Central
limit theorems for additive functionals of Markov chains can be traced back to
the works of Kolmogorov and Doeblin \cite{MR1505091}. The discrete time allows
to decompose the sample paths into excursions. The link with stationary
sequences goes back to Gordin \cite{MR0251785}, see also Ibragimov and Linnik
\cite{MR0202176} and Nagaev \cite{MR0094846} (only stable laws can appear at
the limit). The link with martingales goes back to Gordin and Lifsic
\cite{MR0501277}. For diffusions, the martingale method was developed by
Kipnis and Varadhan \cite{KV}, see also \cite{Helland} (the Poisson equation
is solved via the resolvent).

\textbf{Outline.} Section \ref{secpreliminaires} provides some notations and
preliminaries including a discussion on the variance of $S_t$. Section
\ref{strategy} is devoted to FCLT at equilibrium and contains a lot of known
results. We recall how to use the Poisson equation and compare with the known
results on stationary sequences, which seems more powerful. In particular, we
give in section \ref{subsecKV} a direct new proof of the renowned FCLT of
Kipnis and Varadhan \cite[Corollary 1.9]{KV} in the reversible case. In
section \ref{subsecKVnonS} we provide a non-reversible version of the
Kipnis-Varadhan theorem. Actually some of the results of section \ref{seccomp}
are written in the CLT situation, but under mild assumptions, they can be
extended to a general FCLT (see Proposition \ref{pr:cltinv}). All these
general results are illustrated by the examples discussed in Section
\ref{secexamples}. In sections \ref{secexanomal} and \ref{secanomalous} we
exhibit a particularly interesting behavior, i.e. a possible anomalous rate of
convergence to a Gaussian limit. This behavior is a consequence of a not too
slow decay to equilibrium in the ergodic theorem. Finally we give in the next
section some results concerning fluctuations out of equilibrium.

\subsection*{Acknowledgments}

This work benefited from discussions with N. Ben Abdallah, M. Puel and S.
Motsch, in the Institut de Math\'ematiques de Toulouse.


\section{The framework}
\label{secpreliminaires}

Unless otherwise stated ${(X_t)}_{t\geq0}$ is a continuous time strong Markov
process with state space $\dR^d$, non explosive, irreducible, positive
recurrent, with unique invariant probability measure $\mu$. We realize the
process on a canonical space and we denote by $\dP_\nu$ the law of the process
with initial law $\nu=\cL(X_0)$. In particular
$\dP_x:=\dP_{\delta_x}=\cL((X_t)_{t\geq0}|X_0=x)$ for all $x\in E$. We denote
by $\dE_\nu$ and $\Var_\nu$ the expectation and variance under $\dP_\nu$. For
all $t\geq0$, all $x\in E$, and every $f:E\to\dR$ integrable for
$\cL(X_t|X_0=x)$, we define the function $P_t(f):x\mapsto\dE(f(X_t)|X_0=x)$.
One can check that $P_t(f)$ is well defined for all $f:E\to\dR$ which is
measurable and positive, or in $\dL^p(\mu)$ for $1\leq p\leq\infty$. On each
$\dL^p(\mu)$ with $1\leq p\leq \infty$, the family $(P_t)_{t\geq0}$ forms a
Markov semigroup of linear operators of unit norm, leaving stable each
constant function and preserving globally the set of non negative functions.
We denote by $L$ the infinitesimal generator of this semigroup in
$\dL^2(\mu)$, defined by $Lf:=\lim_{t\to0}t^{-1}(P_t(f)-f)$. We assume that
$(X_t)_{t\geq0}$ is a diffusion process (this implies that for all $x\in E$
the law $\dP_x$ is supported in the set of continuous functions from $\dR_+$
to $\dR^d$ taking the value $x$ at time $0$) and that there exists an algebra
$\dD(L)$ of uniformly continuous and bounded functions, containing constant
functions, which is a core for the extended domain $\dD_e(L)$ of the
generator, see e.g.\ \cite{CL96, DM4}. Following \cite{CL96}, one can then
show that there exists a countable orthogonal family $(C^n)$ of local
martingales and a countable family $(\na^n)$ of operators such that for all
$f\in\dD_e(L)$, the stochastic process ${(M_t)}_{t\geq0}$ defined from $f$ by
\begin{equation}\label{1.1}
  M_t %
  :=f(X_t)-f(X_0)-\int_0^t\! Lf(X_s) \, ds %
  =\sum_n \int_0^t\!\na^n f(X_s)\,dC^n_s,
\end{equation}
is a square integrable local martingale for all probability measure on $E$.
Its bracket is
\[
\DP{M}_t =\int_0^t\!\Ga(f)(X_s) \, ds.
\] 
where $\Ga(f)$ is the carr\'e-du-champ functional quadratic form defined for
any $f\in\dD(L)$ by
\begin{equation}\label{eqcarre}
  \Ga(f):=\sum_n \na^n f\,\na^n f.
\end{equation}
We write for convenience $M_t= \int_0^t\!\na f(X_s)\,dC_s$.
With these definitions, for $f\in \dD(L)$,
\begin{equation}\label{1.2}
  \cE(f)
  :=\int\!\Ga(f)\,d\mu %
  =-2\int\!f\,Lf\,d\mu \, %
  =- \pd_{t=0}\NRM{P_tf}^2_{\dL^2(\mu)}.
\end{equation}
The diffusion property states that for every smooth $\Phi:\dR^n\to\dR$ and
$f_1,\ldots,f_n\in\dD(L)$,
\[
L\Phi(f_1,\dots,f_n) %
= \sum_{i=1}^n \, \frac{\pd \Phi}{\pd x_i}(f_1,\dots,f_n) \, Lf_i %
+ \frac12 \, \sum_{i,j=1}^n \, %
\frac{\pd^2 \Phi}{\pd x_i \, \pd x_j}(f_1,\dots,f_n) \,
\Ga(f_i,f_j)
\]
where $\Ga(f,g)=L \, (fg) - f \, Lg -g \, Lf$ is the bilinear form associated
to the carr\'e-du-champ. We shall also use the adjoint $L^*$ of $L$ in
$\dL^2(\mu)$ given for all $f,g\in\dD(L)$ by
\[
\int\!fLg\,d\mu = \int\!gL^*f\,d\mu
\]
and the corresponding semigroup $(P_t^*)_{t\geq0}$. We shall mainly be
interested by diffusion processes with generator of the form
\begin{equation}\label{eqgenediff}
  L = \frac{1}{2}\sum_{i,j=1}^d \, A_{ij}(x) \, \pd^2_{i,j} \, %
  + \, \sum_{i=1}^d \, B_i(x) \, \pd_i
\end{equation}
where $x\mapsto A(x):=(A_{i,j}(x))_{1\leq i,j\leq d}$ is a smooth field of
symmetric positive semidefinite matrices, and $x\mapsto b(x):=(b_i(x))_{1\leq
  i\leq d}$ is a smooth vector field. If we denote by $(X_t^x)_{t\geq0}$ a
process of law $\dP_x$ then it is the solution of the stochastic differential
equation
\begin{equation}\label{eqsde}
  dX_t^x = b(X_t^x)\,dt + \sqrt{A}(X_t^x)dB_t,
  \quad\text{with}\quad  X_0^x = x
\end{equation}
where $(B_t)_{t\geq0}$ is a $d$-dimensional standard Brownian Motion, and we
have also
\[
\Ga(f) = \DP{A\na f,\na f}.
\] 
Note that since the process admits a unique invariant probability measure
$\mu$, the process is positive recurrent. We say that the invariant
probability measure $\mu$ is \emph{reversible} when $L=L^*$ (and thus
$P_t=P_t^*$ for all $t\geq0$).

In practice, the initial data consists in the operator $L$. We give below a
criterion on $L$ ensuring the existence of a unique probability measure and
thus positive recurrence.

\begin{definition}[Lyapunov function]\label{df:Lyap}
  Let $\vphi:[1,+\infty[\to\,]0,\infty[$. We say that $V \in D_e(L)$ (the
  extended domain of the generator, see \cite{CL96, DM4}) is a
  $\vphi$-Lyapunov function if $V \geq 1$ and if there exist a constant
  $\kappa$ and a closed petite set $C$ such that for all $x$
  \[
  LV(x) \, \leq \, - \, \vphi(V(x)) \, + \, \kappa \, \mathbf{1}_C(x) \, .
  \]
  Recall that $C$ is a \underline{petite set} if there exists some probability
  measure $p(dt)$ on $\dR_+$ such that for all $x \in C$ ,
  $\int_0^{\infty}\!P_t(x,\cdot)\,p(dt) \geq \nu$ for a non trivial positive
  measure $\nu$.
\end{definition}

In the $\dR^d$ situation with $L$ given by \eqref{eqgenediff} with smooth
coefficients, compact subsets are petite sets and we have 
the following \cite{Has80}:

\begin{proposition}\label{pr:rec}
  If $L$ is given by \eqref{eqgenediff} a sufficient condition for positive
  recurrence is the existence of a $\vphi$-Lyapunov function with $\vphi(u)=1$
  and for $C$ some compact subset. In addition, for all $x\in \dR^d$ the law of
  \eqref{eqsde} denoted by $P_t(x,.)$ converges to the unique invariant
  probability measure $\mu$ in total variation distance, as $t \to +\infty$.
\end{proposition}
\smallskip

We say that an invariant probability measure $\mu$ is \underline{ergodic} if
the only invariant functions (i.e. such that $P_tf=f$ for all $t$) are the
constants. In this case the ergodic theorem says that the Ces\`aro means
$\frac{1}{t}\int_0^t\!f(X_s)\,ds$ converge, as $t\to\infty$, $\dP_\mu$ almost
surely and in $\dL^1$, to $\int\!f\,d\mu$ for any $f\in\dL^1(\mu)$. We say
that the process is \underline{strongly ergodic} if $P_tf\to\int\!f\,d\mu$ in
$\dL^2(\mu)$ for any $f\in \dL^2(\mu)$ (this immediately extends to
$\dL^p(\mu)$, $1\leq p<+\infty$) and recall that $t \mapsto \NRM{P_t
  f}_{\dL^2(\mu)}$ is always non increasing. If $\mu$ is ergodic and
reversible then the process is strongly ergodic. We say that the Dirichlet
form is non degenerate if $\cE(f,f)=0$ if and only if $f$ is constant. Again the
reversible ergodic case is non degenerate, but kinetic models will be
degenerate. We refer to section 5 in \cite{Cat04} for a detailed discussion of
these notions.

\begin{lemma}[Variance in the reversible case]\label{factsym}
  Assume that $\mu$ is reversible and $0\neq f\in\dL^2(\mu)$ with
  $\int\!f\,d\mu=0$. Then we have the following properties:
  \begin{enumerate}
  \item $\liminf_{t\to\infty} \frac{1}{t}\,\Var_\mu(S_t)>0$
  \item $\limsup_{t\to\infty}\frac{1}{t}\,\Var_\mu(S_t)<\infty$ iff 
    the Kipnis-Varadhan condition is satisfied:
    \begin{equation}\label{eqKVcond}
      V:= \int_0^{\infty}\!\PAR{\int\!(P_sf)^2\,d\mu}\,ds <\infty,
    \end{equation}
    and in this case $\lim_{t\to\infty}\frac{1}{t}\,\Var_\mu(S_t) = 4V$
  \end{enumerate}
\end{lemma}

The quantity $4V$ is the \emph{asymptotic variance of the scaled additive
  functional $\frac{1}{t}S_t$}.

\begin{proof}
  By using the Markov property, and the invariance of $\mu$, we can write
  \begin{align*}
    \Var_\mu(S_t)
    & = \dE(S^2_t)\\
    & = 2\int_{0\leq u\leq s\leq t}\!\dE[f(X_s)f(X_u)]\,duds\\
    & =  2\int_{0\leq u\leq s\leq t}\!\PAR{\int\!fP_{s-u}f\,d\mu}duds\\
    & = 2\int_{0\leq u\leq s\leq t}\!\PAR{\int\!fP_uf\,d\mu}\,duds\\
    & = 2\int_{0\leq u\leq s\leq t}\!\PAR{\int\!P^*_{u/2}fP_{u/2} f\,d\mu}duds\\
    & = 4\int_0^{t/2}\!(t-2s)\PAR{\int\!P^*_sfP_sf\,d\mu}\,ds.
  \end{align*}
  Using now the reversibility of $\mu$ and the decay of the $\dL^2$ norm, we
  obtain
  \[
  2t\int_0^{t/4}\!\PAR{\int\!(P_sf)^2\,d\mu}\,ds\, %
  \leq \Var_\mu(S_t) %
  \leq 4t\int_0^{t/2}\!\PAR{\int\!(P_sf)^2\,d\mu}\,ds.
  \]
  This implies the first property. The second property follows from the
  Ces\`aro rule and
  \[
  \frac{\Var_\mu(S_t)}{t} %
  = \frac{2}{t}\int_{0\leq u\leq s\leq t}\!\PAR{\int\!P^2_{u/2}f\,d\mu}du\,ds.
  \]
\end{proof}

\begin{remark}[Non reversible case]
  If $\mu$ is not reversible, we do not even know whether $\int\!P^*_sfP_s
  f\,d\mu$ is non-negative or not. Nevertheless we may define $V_-$ and $V_+$
  by
  \[ 
  V_- := \liminf_{t\to\infty}\int_0^t\!\PAR{\int\!P_sfP^*_sf\,d\mu}\,ds
  \quad\text{and}\quad
  V_+ := \limsup_{t\to\infty}\int_0^t\!\PAR{\int\!P_sfP^*_sf\,d\mu}\,ds
  \] 
  abridged into $V$ if $V_+=V_-$. As in the reversible case, if $V_+<+\infty$
  then $V_+=V_-$ and $\lim_{t\to\infty}t^{-1}\Var_\mu(S_t)= 4V$. We
  ignore if $V_-(f)>0$ as in the reversible case. We have thus a priori to
  face two type of situations: either $V_+<+\infty$ and the asymptotic
  variance exists and $\Var_\mu(S_t)$ is of order $t$ as $t\to\infty$,
  or $V_+=+\infty$ and $\Var_\mu(S_t)$ is much larger.
\end{remark}

\begin{remark}[Possible limits]
  For every sequence $(\nu_n)_{n\geq1}$ of probability measure on $\dR$ with
  unit second moment and zero mean, it can be shown by using for instance the
  Skorokhod representation theorem that all adherence values of
  $(\nu_n)_{n\geq1}$ for the weak topology (with respect to continuous bounded
  functions) have second moment $\leq 1$ and mean $0$. In particular, if an
  adherence value is a stable law then it is necessarily a centered Gaussian
  with variance $\leq 1$. As a consequence, if
  ${(S_t/\sqrt{\Var_\mu(S_t)})}_{t\geq0}$ converges in law to a
  probability measure as $t\to\infty$, then this probability measure has
  second moment $\leq 1$ and mean $0$, and if it is a stable law, then it is a
  centered Gaussian with variance $\leq 1$.
\end{remark}



\section{Poisson equation and martingale approximation}
\label{strategy}


We present in this section a strategy to prove \eqref{eq:FCLT} which consists
in a reduction to a more standard result for a family of martingales. We start
by solving the Poisson equation: we fix $0\neq f\in\dL^2(\mu)$,
$\int\!f\,d\mu=0$, and we seek for $g$ solving
\begin{equation}\label{eq:poisson}
Lg=f.
\end{equation}
The Poisson equation \eqref{eq:poisson} corresponds to a so called
\emph{coboundary} in ergodic theory. If \eqref{eq:poisson} admits a regular
enough solution $g$, then by It\^o's formula, for every $t\geq0$ and $\ep>0$,
\begin{equation}\label{eqito}
  S_{\ep^{-1}t} %
  = \int_0^{\ep^{-1}t}\!f(X_s)\,ds %
  = g(X_{\ep^{-1}t}) - g(X_0) - M_t^\ep
\end{equation}
where ${(M_t^\ep)}_{t\geq0}$ is a local martingale with brackets
\begin{equation}\label{eqbracket}
  \DP{M^\ep}_t = \int_0^{\ep^{-1}t}\!\Ga(g)(X_s)\,ds.
\end{equation}
Now the Rebolledo FCLT for $\dL^2$ local martingales (see \cite{Reb} or
\cite{MR2368952}) says that if
\begin{equation}\label{eqrebol1}
  v^2(\ep) \DP{M^\ep}_t
  \underset{\ep\to0}{\overset{\dP}{\longrightarrow}} 
  h^2(t)
\end{equation}
for all $t\geq0$, where $v$ and $h$ are deterministic functions which may
depend on $f$ via $g$, then
\begin{equation}\label{eqrebol2}
  \PAR{v(\ep)M^\ep_t}_{t\geq0}
  \underset{\ep\to0}{\overset{\text{Law}}{\longrightarrow}}
  \PAR{\int_0^t\,h(s)\,dW_s}_{t\geq0}
\end{equation}
where ${(W_t)}_{t\geq0}$ is a standard Brownian Motion, the convergence in law
being in the sense of finite dimensional process marginal laws. To obtain
\eqref{eq:FCLT}, it suffices to show the convergence in probability to $0$ of
$v(\ep)g(X_{\ep^{-1} t})$ as $\ep\to0$, for any fixed $t\geq0$. Moreover, if
this convergence holds in $\dL^2$ then the normalization factor $v$ can be
chosen such that
\begin{equation}\label{eqasympvar}
  \lim_{\ep\to0}v^2(\ep)\dE\SBRA{S_{\ep^{-1}t}^2} %
  = \lim_{\ep\to0}v^2(\ep)\dE\SBRA{\DP{M^\ep}_t} %
  = \lim_{\ep\to0}v^2(\ep)\frac{t}{\ep}\cE(g) %
  = h^2(t)
\end{equation}
i.e.\ we recover $v(\ep)=\sqrt{\ep}$ and
$V=\lim_{t\to\infty}t^{-1}\Var_\mu(S_t)=\frac{1}{4}\cE(g)$. To summarize, this
martingale approach reduces the proof of \eqref{eq:FCLT} to the following
three steps:
\begin{itemize}
\item solve the Poisson equation $Lg=f$ in the $g$ variable
\item control the regularity of $g$ in order to use It\^o's formula
  \eqref{eqito}
\item check the convergence to $0$ of $g(X_{\ep^{-1}t})$ as $\ep\to0$ in an
  appropriate way.
\end{itemize}

Let us start with a simple proposition which follows from the discussion above.

\begin{theorem}[FCLT via Poisson equation in $\dL^2$]\label{th:poissonfacile}
  If $0\neq f\in\dL^2(\mu)$ with $\int\!f\,d\mu=0$, and if $f \in \dD(L^{-1})$
  i.e.\ there exists $g \in \dD(L)$ such that $Lg=f$ where $L$ is seen as an
  unbounded operator, then $\Var_\mu(S_t)\sim_{t\to\infty}t\cE(g,g)$ and
  \eqref{eq:FCLT} holds under $\dP_\mu$ with $s_t^2(f)=t\cE(g,g)$.
\end{theorem}

Let us examine a natural candidate to solve the Poisson equation. Assume that $Lg=f$
in $\dL^2(\mu)$ and that $\int\!g\,d\mu =0$ (note that since $L1=0$ we may
always center $g$). Then
\[
P_tg-g %
= \int_0^t\!\pd_sP_sg\,ds %
= \int_0^t\!LP_s g\,ds %
= \int_0^t\!P_sLg\,ds %
=\int_0^t\,P_sf\,ds
\] 
so that, if the process is strongly ergodic,
$\lim_{t\to\infty}P_tg=\int\!g\,d\mu=0$, and thus
\begin{equation}\label{eq:g}
  g = -\int_0^{\infty}\!P_sf\,ds.
\end{equation}
For the latter to be well defined in $\dL^2(\mu)$, it is enough to have some
quantitative controls for the convergence of $P_sf$ to $0$ as $s\to\infty$.
Conversely, for a deterministic $T>0$ we set
\begin{equation}\label{eqgt}
  g_T := -\int_0^T\!P_sf\,ds
\end{equation}
which is well defined in $\dL^2(\mu)$ and satisfies to
\begin{equation*}\label{eqgtbis}
  Lg_T = \lim_{u\to0}\frac{P_ug_T-g_T}{u} = -\partial_{u=0}\int_{u}^{u+T}\!P_sf\,ds=f-P_Tf.
\end{equation*}
If $g_T$ converges in $\dL^2$ to $g$ then $Lg = f$. In particular, we obtain
the following.

\begin{corollary}[Solving the Poisson equation in $\dL^2$]\label{co:facile}
  Let $0\neq f\in\dL^2(\mu)$ with $\int\!f\,d\mu=0$.
  \begin{enumerate}
  \item If we have
    \begin{equation}\label{eqfini}
      \int_0^{\infty}\!s\NRM{P_sf}_{\dL^2(\mu)}\,ds<\infty,
    \end{equation}
    then $f\in\dD(L^{-1})$ and $g$ in \eqref{eq:g} is in $\dL^2(\mu)$ and
    solves the Poisson equation \eqref{eq:poisson}
  \item If $\mu$ is reversible then $f \in D(L^{-1})$ if and only if
    \begin{equation}\label{eqfinibis}
      \int_0^{\infty}\!s\NRM{P_s f}_{\dL^2(\mu)}^2\,ds<\infty,
    \end{equation}
    and in this case the Poisson equation \eqref{eq:poisson} has a unique
    solution $g$ given by \eqref{eq:g}.
  \end{enumerate}
  Moreover, condition \eqref{eqfini} implies condition \eqref{eqfinibis}.
\end{corollary}

\begin{proof}
  The existence of $g\in \dL^2(\mu)$ in the case \eqref{eqfini} is immediate.
  For \eqref{eqfinibis} consider $g_T$ defined in \eqref{eqgt}. For $a>0$ we
  then have, using reversibility
  \begin{align*}
    \int\!|g_{T+a} - g_T|^2\,d\mu %
    &= 2\int\!\PAR{\int_T^{T+a}\! P_sf\,\int_T^s\!P_uf\,du\,ds}\,d\mu \\
    &= 2\int\!\PAR{\int_T^{T+a}\!\int_T^s\!\PAR{P_{\frac{s+u}{2}}f}^2\,du\,ds}\,d\mu \\
    &= 4\int\!\PAR{\int_T^{T+a}\!(u-T)\,(P_uf)^2\,du}\,d\mu,
  \end{align*}
  so that $(g_T)_T$ is Cauchy, hence convergent, if and only if
  \eqref{eqfinibis} is satisfied. In addition, taking $T=0$ above gives
  \[
  \int\!g_T^2\,d\mu = 4\int_0^T\!u\,\PAR{\int\!(P_uf)^2\,d\mu}\,du.
  \]
  Hence the family $(g_T)_T$ is bounded in $\dL^2$ only if \eqref{eqfinibis}
  is satisfied, i.e. here convergence and boundedness of $(g_T)_T$ are
  equivalent.

  To deduce \eqref{eqfinibis} from \eqref{eqfini}, we note that $t \mapsto
  \NRM{P_tf}_{\dL^2(\mu)}$ is non-increasing, and hence,
  \[
  t \, \NRM{P_t f}_{\dL^2(\mu)} %
  \leq \int_0^t \, \NRM{P_s f}_{\dL^2(\mu)} ds %
  \leq \int_0^{\infty} \NRM{P_s f}_{\dL^2(\mu)} ds
  \]
  so that $\NRM{P_t f}_{\dL^2(\mu)} =O(1/t)$ by \eqref{eqfini}, which gives
  \eqref{eqfinibis}. We remark by the way that conversely, \eqref{eqfinibis}
  implies $\NRM{P_tf}_{\dL^2(\mu)} =O(1/t)$ since by the same reasoning,
  \[
  \frac{1}{2}\,t^2 \, \NRM{P_t f}^2_{\dL^2(\mu)} %
  \leq \int_0^{+\infty}\!s\,\NRM{P_sf}_{\dL^2(\mu)}^2\,ds.
  \] 
\end{proof}

Recent results on the asymptotic behavior of such semigroups can be used to
give tractable conditions and general examples. We shall recall them later. In
particular for $\dR^d$ valued diffusion processes we will compare them with
\cite{GM,PV1,PV2,PV3}.

Actually one can (partly) improve on this result. For instance if $\mu$ is a
reversible measure, the same FCLT holds under the weaker assumption $f \in
\dD(L^{-1/2})$ as shown in \cite{KV} and revisited in the next subsection too.
For non-reversible Markov chains, a systematic study of fractional Poisson
equation is done in \cite{DL2}. The connection with the rate of convergence of
$P_t f$ is also discussed therein, and the result ``at equilibrium'' is
extended to an initial $\de_x$ Dirac mass in \cite{DL1,DL3} extending
\cite{MW} for the central limit theorem (i.e. for each marginal of the
process). The previous $f\in \dD(L^{-1/2})$ is however no more sufficient (see
the final discussion in \cite{DL3}). It is thus more natural to look at the
rate of convergence (as in \cite{DL3,MW}) rather than at fractional operators.
\medskip

\subsection{Reversible case and Kipnis-Varadhan theorem}
\label{subsecKV}

In this section we assume that $\mu$ is reversible. Corollary \ref{co:facile}
states that \eqref{eqKVcond} (equivalent to the existence of the asymptotic
variance) is not sufficient to solve the Poisson equation, even in a weak
sense. Nevertheless it is enough to get \eqref{eq:FCLT}, the result below is
Corollary 1.9 of \cite{KV}.

\begin{theorem}[FCLT from the existence of asymptotic variance]\label{th:KV}
  Assume that $\mu$ is reversible, that $0\neq f\in\dL^2(\mu)$ with
  $\int\!f\,d\mu=0$, and that $f$ satisfies the Kipnis-Varadhan condition
  \eqref{eqKVcond}. Then \eqref{eq:FCLT} holds under $\dP_\mu$ with
  $s_t^2=4tV$, and $\Var_\mu(S_t)\sim_{t\to\infty} s_t^2$.
\end{theorem}
\begin{proof}
  For $T>0$ introduce $g_{T}$ by \eqref{eqgt}, and the corresponding family
  $((\na^n g_{T}))_{T>0}$ (recall \eqref{1.1}). We thus have $Lg_{T} = f -
  P_{T} f$ and, for all $S\leq T$,
  \begin{align*}
    \int\!\Ga(g_T-g_S)\,d\mu
    &= 2\int\!(-L(g_T-g_S))\,(g_T-g_S)\,d\mu\\
    &= 2\int_S^T\!\int\!(P_S f - P_Tf)\,P_sf\,d\mu\,ds \\
    &= 2\int_S^T\!\int\!(P_{(s+S)/2}^2f-P_{(s+T)/2}^2f)\,d\mu\,ds \\
    & \leq 4\int_S^{\infty}\!\int\!P_s^2f\,d\mu\,ds,
  \end{align*}
  so that according to \eqref{eqKVcond}, the family $((\na^n g_{T}))_{T>0}$ is
  Cauchy in $\dL^2(\mu)$. It follows that it strongly converges to $h$ in
  $\dL^2(\mu)$. On the other hand, using It\^o's formula,
  \begin{align}\label{eqpapprox}
    S^T_{t/\ep} 
    &= g_{T}(X_{t/\ep})-g_{T}(X_0)-M_t^{T} + \int_0^{t/\ep}\!P_{T}f(X_s)\,ds \\
    &= g_{T}(X_{t/\ep})-g_{T}(X_0)-M_t^{T} + S_{t/\ep}^{T} \nonumber
  \end{align}
  where ${(M_t^{T})}_{t\geq0}$ is a martingale with brackets $\DP{M^{T}}_t =
  \int_0^{t/\ep}\!\Ga(g_{T})(X_s)\,ds$ (recall \eqref{eqcarre}).
  
  According to what precedes and the framework (recall \eqref{1.1}) we may
  replace ${(M_t^T)}_{t\geq0}$ by another martingale ${(N_t^h)}_{t\geq0}$ with
  brackets $\DP{N^{h}}_t = \int_0^{t/\ep}\!|h|^2(X_s)\,ds$ such that 
  \[
  \ep\dE_\mu\PAR{\sup_{0\leq s \leq t} |M_s^T - N_s^h|^2} %
  \leq t\NRM{\na g_T - h}_{\dL^2(\mu)}^2 %
  \to 0 \, \textrm{
    as } T \to\infty \text{\ uniformly in $\ep$}.
  \] 
  In addition the ergodic theorem tells us that
  \[
  \lim_{\ep\to0}\ep\DP{N^{h}}_t = t\int\!h^2\,d\mu.
  \] 
  Thus we may again apply Rebolledo's FCLT, taking first the limit in $T$ and
  then in $\ep$. It remains to control the others terms. But
  \begin{align*}
    \Var_\mu(S^T_{t/\ep}) 
    & = 2\int_0^{t/\ep}\!\int_0^s\!\PAR{P^2_{{T}+(u/2)} f \, d\mu}\,du\,ds \\ 
    & = 4\int_0^{t/\ep}\!\int_{T}^{{T}+(s/2)}\!\PAR{\int\!P^2_{u}f\,d\mu}\,du\,ds\\ 
    & \leq 4\int_0^{t/\ep}\!\int_{T}^{\infty}\!\PAR{\int\!P^2_{u}f\,d\mu}du\,ds\\ 
    & \leq 4(t/\ep)\int_{T}^{\infty}\!\PAR{\int\!P^2_{u}f\,d\mu}\,du.
  \end{align*}
  Since $\lim_{T\to\infty}\int_T^{\infty}\!\PAR{\int\!P^2_{u}f\,d\mu}\,du=0$
  according to \eqref{eqKVcond}, we have, uniformly in $\ep$,
  \[
  \lim_{T\to\infty} \ep\Var_\mu(S_{t/\ep}^{T})=0.
  \]
  Next, 
  \[
  \int\!g_{T}^2\,d\mu %
  = 4\int_0^{T}\!u\PAR{\int\!P^2_{u}f\,d\mu}\,du %
  \leq 4{T}\int_0^{\infty}\!\PAR{\int\!P^2_{u}f\,d\mu}\,du.
  \]
  Hence $\lim_{\ep \to 0}\ep\NRM{g_{T}}_{\dL^2(\mu)}^2=0$. The desired result
  follows by taking $T$ large enough.
\end{proof}

\begin{remark}
  Our proof is different from the original one by Kipnis and Varadhan and is
  perhaps simpler. Indeed we have chosen to use the natural approximation of
  what should be the solution of the Poisson equation (i.e $g_t$), rather than
  the approximating $R_\ep$ resolvent as in \cite{KV}. Let us mention at this
  point the work by Holzmann \cite{Holz} giving a necessary and sufficient
  condition for the so called ``martingale approximation'' property (we get
  some in our proof), thanks to an approximation procedure using the
  resolvent.
\end{remark}

\begin{remark}[By D. Bakry]\label{rembak}
  The condition \eqref{eqKVcond} is satisfied if Assumption (1.14) in
  \cite{KV} is satisfied i.e. there exists a constant $c_f$ such that for all
  $F$ in the domain of $\cE$,
  \begin{equation}\label{eqKV}
    \PAR{\int\!f\,F\,d\mu}^2 \leq -c^2_f\,\int\!FLF\,d\mu.
  \end{equation}
  Indeed, if we define $\vphi(t):=-\int\!f\,g_t\,d\mu$ where as usual
  $g_t=-\int_0^tP_sf\,ds$, and if we take $F=g_t$, then $-LF= -Lg_t=P_tf - f$,
  and using \eqref{eqKV} we get $\vphi^2(t)\leq c_f^2(2\vphi(t)-\vphi(2t))$.
  Using that $\vphi(2t) \geq 0$ we obtain $2c_f^2\vphi(t)-\vphi^2(t)\geq0$
  which implies that $\vphi$ is bounded hence $\vphi(+\infty)<+\infty$. Taking
  the limit as $t\to\infty$ and using $2V(f)=\vphi(+\infty)$, we obtain
  \[
  V(f) \leq \frac{1}{2}c_f^2.
  \]
  All this can be interpreted in terms of the domain of $(-L)^{-1/2}$ (which
  is formally the gradient $\na$) i.e. condition \eqref{eqKVcond} can be seen
  to be equivalent to the existence in $\dL^2(\mu)$ of
  \[
  (-L)^{-1/2}f =c\int_0^{\infty}\!s^{-\frac{1}{2}}\,P_sf\,ds
  \] 
  for an ad-hoc constant $c$. Indeed, for some constant $C>0$,
  \[
  \NRM{\int_0^{\infty}\!s^{-\frac 12}\,P_sf\,ds}_{\dL^2(\mu)}^2
  \leq C\int\!
  \int_0^{\infty}\!P^2_sf\,\PAR{\int_s^{2s}\!(2u-s)^{-1/2}\,u^{-1/2}\,du}\,ds\,d\mu 
  \]
  and $\int_s^{2s}\!(2u-s)^{-1/2}\,u^{-1/2}\,du$ is bounded. 
  Note that 
  \eqref{eqKVcond} implies that $\NRM{P_t f}_{\dL^2(\mu)}\leq C(f)/\sqrt t$.
\end{remark}

We shall come back later to the method we used in the previous proof, for more
general situations including anomalous rate of convergence.

\subsection{Poisson equation in $\dL^q$ with $q\leq 2$ for diffusions}
\label{diff}

What has been done before is written in a $\dL^2$ framework. But the method can
be extended to a more general setting. Indeed, what is really needed is
\begin{enumerate}
\item a solution $g \in \dL^q(\mu)$ of the Poisson equation, for some $q\geq
  1$,
\item sufficient smoothness of $g$ in order to apply It\^o's formula, 
\item control the brackets i.e. give a sense to the following
  quantities 
  \[
  \int\!\Ga(g)\,d\mu = -2\int\!f\,g\,d\mu.
  \]
\end{enumerate}

\begin{definition}[Ergodic rate of convergence]\label{df:rate}
  For any $r\geq p\geq 1$ and $t\geq0$ we define
  \[
  t\mapsto
  \al_{p,r}(t):=
  \sup_{\substack{\NRM{g}_{\dL^r(\mu)}=1\\\int\!g\,d\mu=0}}\NRM{P_tg}_{\dL^p(\mu)}.
  \]
  The uniform decay rate is $\al := \al_{2,\infty}$. We denote by $\al^*$ the
  uniform decay rate of $L^*$. We say that the process is \underline{uniformly
    ergodic} if $\lim_{t\to\infty}\al(t)=0$.
\end{definition}

We shall discuss later how to get some estimates on these decay rates.

\begin{proposition}[Solving the Poisson equation in $\dL^q$]\label{pr:poissonlq}
  Let $p\geq2$ and $q:=p/(p-1)$. If
  \[
  f\in\dL^p(\mu)
  \quad\text{and}\quad
  \int\!f\,d\mu=0
  \quad\text{and}\quad
  \int_0^{\infty}\!\al^*_{2,p}(t)\,\NRM{P_tf}_{\dL^2(\mu)}\,dt<\infty
  \] 
  then $g := - \int_0^{\infty}\!P_sf\,ds$ belongs to $\dL^q(\mu)$ and solves
  the Poisson equation $Lg = f$.
\end{proposition}

The assumption of Proposition \ref{pr:poissonlq} is satisfied for any
$\mu$-centered $f\in\dL^p(\mu)$ if
\[
\int_0^{\infty}\!\al^*_{2,p}(t)\al_{2,p}(t)\,dt<\infty.
\]
In the reversible case, we recover a version of the Kipnis-Varadhan statement
implying a stronger result (the existence of a solution of the Poisson
equation). The results of this section are mainly interesting in the
non-reversible situation.

\begin{proof}
  Let $h \in \dL^p(\mu)$, $\bar{h}:=h-\int\!h\,d\mu$, $T>0$ and $g_T
  :=-\int_0^T\!P_tf\,dt$. Then
  \begin{align*}
    \ABS{\int\!h\,(g_{T+a}-g_T)\,d\mu} 
    & =\ABS{\int\!\bar{h}\,(g_{T+a}-g_T)\,d\mu} \\
    &= \ABS{\int_T^{T+a}\!\PAR{\int\!P_{t/2}^*\bar{h}\,P_{t/2}f\,d\mu}\,dt} \\
    &
    \leq \PAR{\int_T^{T+a}\!\al^*_{2,p}(t/2)\,\NRM{P_{t/2}f}_{\dL^2(\mu)}\,dt} %
    \NRM{h}_{\dL^p(\mu)}.
  \end{align*}
  As in the proof of Corollary \ref{co:facile}, $g_T$ is Cauchy, hence
  convergent in $\dL^q(\mu)$ and solves the Poisson equation.
\end{proof}

The previous proof ``by duality'' can be improved, just calculating the
$\dL^q(\mu)$ norm of $g_T$, for some $1\leq q \leq 2$ which is not necessarily
the conjugate of $p$.

\begin{proposition}[Solving the Poisson equation in
  $\dL^q$]\label{pr:poissonlqbis}
  Let $p\geq2$ and $1\leq q\leq 2$.
  If 
  \[
  f\in\dL^p(\mu)
  \quad\text{and}\quad
  \int\!f\,d\mu=0
  \quad\text{and}\quad
  \int_0^{\infty}\!t^{q-1}\,\al^*_{2,p/(q-1)}(t)\,\NRM{P_tf}_{\dL^2(\mu)}\,dt<\infty
  \]
  then $g = - \int_0^{\infty}\!P_sf\,ds$ belongs to $\dL^q(\mu)$ and solves
  the Poisson equation $Lg = f$.
\end{proposition}

\begin{proof}
  We have
  \begin{align*}
    \int\!|g_T|^q\,d\mu %
    & = q\int\! %
    \PAR{\int_0^T\!P_sf\,\PAR{\mathbf{1}_{g_s<0}-\mathbf{1}_{g_s>0}}
      \ABS{\int_0^s\!P_uf\,du}^{q-1}ds}\,d\mu \\
    & \leq q\int_0^T\! %
    \NRM{P_{s/2}f}_{\dL^2(\mu)}\NRM{P^*_{s/2}\bar h_s}_{\dL^2(\mu)}\,ds\\
    & \leq q\int_0^T\! %
    \NRM{P_{s/2}f}_{\dL^2(\mu)}\al^*_{2,m}(s/2)\NRM{\bar h_s}_{\dL^m(\mu)}\,ds
  \end{align*}
  for an arbitrary $m\geq2$, where
  \[
  h_s := \PAR{\mathbf{1}_{g_s<0}-\mathbf{1}_{g_s>0}}\,\ABS{\int_0^s\!P_uf\,du}^{q-1} 
  \quad\text{and}\quad
  \bar h_s:=h_s-\int\!h_s\,d\mu.
  \] 
  It remains to choose the best $m$. But of course $\NRM{\bar h_s}_{\dL^m(\mu)}
  \leq 2 \NRM{h_s}_{\dL^m(\mu)}$ and
  \begin{align*}
    \PAR{\int\!|h_s|^m\,d\mu}^{\frac{1}{m}} 
    & =
    s^{(q-1)}\PAR{\int\!\PAR{\int_0^s\!|P_uf|\,\frac{du}{s}}^{(q-1)m}\,d\mu}^{\frac
      1m} \\ 
    & \leq s^{(q-1)} \PAR{\int\!|f|^{(q-1)m}\,d\mu}^{\frac{1}{m}}.
  \end{align*}
  The best choice is $m=p/(q-1)$. We then proceed as in the proof of
  proposition \ref{pr:poissonlq}.
\end{proof}


In view of FCLT, the main difficulty is to apply It\^o's formula in the non
$\dL^2$ context. Though things can be done in some abstract setting, we shall
restrict ourselves here to the diffusion setting \eqref{eqsde}. For simplicity
again we shall consider rather regular settings.

\begin{proposition}[FCLT via the Poisson equation]\label{pr:poissonfaible}
  Assume that 
  \begin{itemize}
  \item $0\neq f\in\dL^2(\mu)$ with $\int\!f\,d\mu=0$
  \item $L$ is given by \eqref{eqgenediff} with smooth coefficients and is
    hypoelliptic
  \item $\mu$ has positive Lesbegue density $\frac{d\mu}{dx}=e^{-U}$ for some
    locally bounded $U$
  \item $f$ is smooth and belongs to $\dL^p(\mu)$ for some $2\leq p$ 
    and, with, $q=p/(p-1)$,
    \[ 
    \int_0^{\infty}\!\al^*_{2,p}(t)\,\NRM{P_t f}_{\dL^2(\mu)}\,dt<\infty
    \quad\text{or}\quad
    \int_0^{\infty}\!t^{q-1}\,\al^*_{2,p/(q-1)}(t)\,\NRM{P_tf}_{\dL^2(\mu)}\,dt<\infty 
    \]
  \end{itemize}
  then $g :=-\int_0^{\infty}\!P_sf\,ds$ is well defined in $\dL^q(\mu)$, is
  smooth, and solves the Poisson equation $Lg = f$, and hence \eqref{eq:FCLT}
  holds under $\dP_\mu$ with $s_t^2=-t\int\!f\,g\,d\mu$.
\end{proposition}

\begin{proof}
  The only thing to do is to show that $g$ (obtained in proposition
  \ref{pr:poissonlq}) satisfies $Lg=f$ in the Schwartz space of distributions
  $\cD'$. To see the latter just write for $h\in \cD$,
  \[
  \int\!L^*h\,g_T\,d\mu %
  = \int\!h\,Lg_T\,d\mu %
  = \int\!h\,(f-P_Tf)\,d\mu
  \]
  and use that $P_T f$ goes to 0 in $\dL^1(\mu)$. It follows that $e^{-U} g_T$
  converges in $\cD'$ to some Schwartz distribution we may write
  $e^{-U} g$, since $e^{-U}$ is everywhere positive and smooth. Furthermore
  since the adjoint operator of $e^{-U} L^*$ (defined on $\cD$) is
  $e^{-U} L$ (defined on $\cD'$), we get that $g$ solves the Poisson
  equation $Lg=f$ in $\cD'$. Using hypoellipticity, we deduce that $g$ is
  smooth and satisfies $Lg=f$ in the usual sense. Finally \eqref{eq:FCLT}
  follows from the usual strategy, provided $\int\!\Ga(g)\,d\mu$ is finite.
  That is why we have to restrict ourselves (in the second case) to $q$ the
  conjugate of $p$, ensuring that $\int\!|fg|\,d\mu <\infty$.
\end{proof}

\begin{remark}\label{solpoisson}
  If $f\in \dL^p(\mu)$ for some $p\geq 1$ ($f$ being still smooth), one can
  immediately adapt the proof of the previous proposition to show that the
  Poisson equation $Lg=f$ has a solution $g\in \dL^1(\mu)$ as soon as
  $\int_0^{+\infty} \, \al^*_{q,\infty}(t) \, dt < +\infty$. \hfill
  $\diamondsuit$
\end{remark}

In the hypoelliptic context one can go a step further. First of all, as before
we may and will assume that $f$ is of $C^\infty$ class, so that $g_t$ is also
smooth. Next, if $\vphi \in \cD(\dR^d)$,
\[
\int \, Lg_t \, \vphi \, p \, dx %
= \int \, Lg_t \, \vphi \, d\mu %
\rightarrow_{t \to +\infty} \int \, f \, \vphi \, d\mu %
= \int \, f \, \vphi \, p \, dx
\] 
so that $p \, Lg_t \, \rightarrow_{t \to +\infty} \, p \, f$ in $\cD'(\dR^d)$,
hence $Lg_t \, \rightarrow_{t \to +\infty} \, f$ in $\cD'(\dR^d)$, since $p$
is smooth and positive.

Assume in addition that there exists a solution $\psi \in \dL^2(\mu)$ of the
Poisson equation $L^*\psi = \vphi$. Thanks to the assumptions, $\psi$ belongs
to $C^\infty$ and solves the Poisson equation in the usual sense. Hence \[\int
\, g_t \, \vphi \, d\mu = \int \, g_t \, L^*\psi \, d\mu = \int \, Lg_t \,
\psi \, d\mu \, \rightarrow_{t \to +\infty} \, \int \, f \, \psi \, d\mu \,
.\] It follows that for every $\vphi \in \cD(\dR^d)$,
\[\langle p \, g_t \, , \, \vphi\rangle \, \rightarrow_{t \to
  +\infty} \, a(\vphi)= \int \, f \, \psi \, d\mu\] where the bracket denotes
the duality bracket between $\cD'(\dR^d)$ and $\cD(\dR^d)$. Thanks to the
uniform boundedness principle it follows that there exists an element $\nu \in
\cD'(\dR^d)$ such that $p \, g_t \to \nu$ in $\cD'(\dR^d)$, and using again
smoothness and positivity of $p$, we have that $g_t \to g = \nu/p$. We
immediately deduce that $Lg=f$ in $\cD'(\dR^d)$, hence thanks to (H3) that
$g\in C^\infty$. Let us summarize all this

\begin{lemma}\label{lempoidistrib}
  Consider the assumptions of proposition \ref{pr:poissonfaible} and assume
  that for all $\vphi \in \cD(\dR^d)$ there exists a solution $\psi \in
  \dL^2(\mu)$ of the Poisson equation $L^*\psi = \vphi$. Then for all smooth
  $f$ there exists some smooth function $g$ such that $Lg=f$.
\end{lemma}

Of course in the cases we are interested in, $g$ does not belong to
$\dL^q(\mu)$ if $f\in \dL^p(\mu)$, so that we cannot use previous results. We
shall give sufficient conditions ensuring that the dual Poisson equation has a
solution for all smooth functions with compact support (see Theorem
\ref{th:GM} in section \ref{secexamples}).

\begin{remark}[The Kipnis Varadhan situation]\label{remkvcauchy}
  If $\vphi \in \cD(\dR)$, we thus have 
  \[
  \int\!f\vphi\,d\mu %
  = \int\!Lg\vphi\,d\mu %
  = \int\!\na g\na\vphi\,d\mu %
  \leq \PAR{\int\!\ABS{\na g}^2\,d\mu}^{\frac{1}{2}} 
  \ABS{\int\!\ABS{\na\vphi}^2\,d\mu}^{\frac{1}{2}}
  \] 
  so that \eqref{eqKV} is satisfied as soon as $\na g \in \dL^2(\mu)$, since
  $\cD(\dR)$ is everywhere dense in $\dL^2(\mu)$.
\end{remark}

\begin{remark}[Time reversal, duality, forward-backward martingale
  decomposition]\label{remarkdual}
  We have just seen that it could be useful to work with $L^*$ too. Actually
  if the process is strongly ergodic, we do not know whether $\lim_{t \to
    +\infty} P_t^*f =0$ for centered $f$'s or not (the limit taking place in
  the $\dL^2$ strong sense). However if the process is uniformly ergodic (i.e.
  $\lim_{t \to +\infty} \al(t)=0$ recall definition \ref{df:rate}) then
  $\lim_{t \to +\infty} \al^*(t)=0$, as will be shown in Proposition
  \ref{pr:mix} in section \ref{seccomp}. Now remark that:
  \[
  \int_0^t \, f(X_s) \, ds = \int_0^t \, f(X_{t-s}) \, ds \, .
  \]
  Since the infinitesimal generator of the process $s \mapsto X_{t-s}$ (for $s
  \leq t$) is given by $L^*$ we can use the previous strategy replacing $L$ by
  $L^*$ and the process $X_.$ by its time reversal up to time $t$.
  It is then known that, similarly to the standard forward decomposition
  \eqref{1.1}, one can associate a backward one
  \begin{equation}\label{1.1.b}
    g(X_0) - g(X_t) - (M^*)_t = \int_0^t \, L^*g(X_s) \, ds \, ,
  \end{equation}
  where $\left((M^*)_{t} - (M^*)_{t-s}\right)_{0 \leq s \leq t}$ is a
  backward martingale with the same brackets as $M$ (in the reversible case
  this is just the time reversal of $M$). The solution to the dual Poisson
  equation $L^*g=f$ thus furnishes a triangular array of local martingales to
  which Rebolledo's FCLT applies. Thus, {\bf all the results we have shown
    with the solution of the Poisson equation are still true with the dual
    Poisson equation, at least in the uniformly ergodic case}. The previous
  remark yields another possible improvement, which is a standard tool in the
  reversible case, namely the so called Lyons-Zheng decomposition. If $g$ is
  smooth enough, summing up the standard forward decomposition \eqref{1.1} and
  the backward decomposition \eqref{1.1.b}, we obtain the forward-backward
  decomposition
  \[
  \int_0^t \,  (L+L^*)g(X_s) \, ds = - \left(M_t + (M^*)_t\right)
  \]
  so that if one can solve the Poisson equation for the symmetrized operator
  $L^S :=L+L^*$ the previous decomposition can be used to study the behavior of
  our additive functional. This is done in e.g. \cite{wufb}, but of course what
  can be obtained is only a tightness result since the addition is not
  compatible with convergence in distribution. However, the forward-backward
  decomposition will be useful in the sequel. 
\end{remark}  

\section{Comparison with general results on stationary sequences}
\label{seccomp}

The CLT and FCLT theory for stationary sequences can be used in our context.
Indeed, let us assume as usual that $X_0\sim\mu$, $0\neq f\in\dL^2(\mu)$,
$\int\!f\,d\mu=0$. We may introduce the stationary sequence of random
variables ${(Y_n)}_{n\geq0}$:
\begin{equation}\label{eqY}
  Y_n := \int_n^{n+1}\!f(X_s)\,ds.  
\end{equation}
and the partial sum $S_n := \sum_{k=0}^{n-1} Y_k$. If $f\in \dL^1(\mu)$ and
$\be(t) \to 0$ as $t \to +\infty$, denoting by $[t]$ the integer part of $t$,
we have that $\be(t) \, \int_{[t]}^t \, f(X_s) \, ds \to 0$ in $\dP_\mu$
probability as $t \to +\infty$, so that the control of the law of our additive
functional reduces to the one of $S_n$ as $n \to +\infty$. We may thus use the
known results for convergence of sums of stationary sequences.

At the process level we may similarly consider the random variables $S_{[nt]}$
where $[\cdot]$ denotes the integer part again, and for $n\leq (1/\ep)<
(n+1)$. The remainder $S_{t/\ep} - S_{[nt]}$ multiplied by a quantity going to
0 will converge to 0 in probability, so that for any $k$-uple of times
$t_1,\ldots,t_k$ we will obtain the convergence (in distribution) of the
corresponding $k$-uple, provided the usual FCLT holds for $S_{[nt]}$.

Hence we may apply the main results in \cite{MPU} for instance. In particular
a renowned result of Maxwell and Woodroofe (\cite{MW} and (18) in \cite{MPU})
adapted to the present situation tells us that \eqref{eq:CLT} holds under
$\dP_\mu$ as soon as $0\neq f\in\dL^2(\mu)$ with $\int\!f\,d\mu=0$ and
\begin{equation}\label{eqMW}
  \int_1^{\infty}\!t^{-\frac 32}\,\PAR{\int\!\PAR{\int_0^t\!P_sf\,ds}^2\,
    d\mu}^{\frac 12}\,dt<\infty.
\end{equation}
This has been improved for chains \cite{Cunylin}. For \eqref{eq:FCLT} we
recall \cite[Cor. 12]{MPU}:

\begin{theorem}[FCLT]\label{th:invMPU}
  Assume that $0\neq f\in\dL^2(\mu)$ with $\int\!f\,d\mu=0$ and that
  \begin{equation}\label{eqcorMW}
    \int_1^{\infty}\!t^{-\frac 12}\,\NRM{P_tf}_{\dL^2(\mu)}\,dt<\infty.
  \end{equation}
  Then \eqref{eq:FCLT} holds true under $\dP_\mu$ with $s_t^2:=\Var_\mu(S_t)$ and
  $s^2:=\lim_{t\to\infty}\frac{1}{t}s_t^2$ exists.
\end{theorem}

Condition \eqref{eqcorMW} is much better than both \eqref{eqfini} and
\eqref{eqfinibis} when $P_tf$ goes slowly to 0. In the reversible case
however, \eqref{eqcorMW} is stronger that the Kipnis-Varadhan condition
\eqref{eqKVcond} (if one prefers Theorem \ref{th:invMPU} is implied by Theorem
\ref{th:KV}), according to what we said in Remark \ref{rembak}. Also note that
in full generality it is worse than the one in Proposition \ref{pr:poissonlq}
as soon as $\al^*_{2,p}(t) \leq c/\sqrt t$ and $f\in \dL^p$. Additionally, an
advantage of the previous section is the simplicity of proofs, compared with
the intricate block decomposition used in the proof of the CLT for general
stationary sequences. 

\subsection{Mixing}
\label{secmix}

Following \cite{CatGui2} (Section 3, Proposition 3.4), let $\cF_s$ (resp.
$\cG_s$) be the $\sig$-field generated by $(X_u)_{u \leq s}$ (resp. $(X_u)_{u
  \geq s}$ ). The strong mixing coefficient $\al_{mix}(r)$ is
\[
\al_{mix}(r) = \sup_{s,F,G} \{|\mathrm{Cov}(F,G)| \}
\]
where the sup runs over $s$ and $F$ (resp. $G$) $\cF_s$ (resp. $\cG_{s+r}$)
measurable, non-negative and bounded by 1. If $\lim_{r \to\infty} \al_{mix}(r)
=0$ then we say that the process is strongly mixing.

\begin{proposition}\label{pr:mix}
  Let $\al$ be as in definition \ref{df:rate}. The following correspondence
  holds : 
  \[
  \al^2(t)\vee (\al^*)^2(t) \leq \al_{mix}(t) \leq \al(t/2) \al^*(t/2).
  \] 
  Hence the process is strongly mixing if and only if it is uniformly ergodic
  (or equivalently if and only if its dual is uniformly ergodic).
\end{proposition}
\begin{proof}
  For the first inequality, it suffices to take $F=P_r f(X_0)$ and $G=f(X_r)$
  (respectively $F=f(X_0)$ and $G=P_r^* f(X_r)$) for $f$ $\mu$-centered and
  bounded by $1$. For the second inequality, let $F$ and $G$ be centered and
  bounded by $1$, respectively $\cF_s$ and $\cG_{s+r}$ measurable. We may
  apply the Markov property to get
  \[
  \dE_\mu[FG] = \dE_\mu [F \, \dE_\mu[G|X_{s+r}]] = \dE_\mu[F \, P_r g(X_s)] 
  \]
  where $g$ is $\mu$-centered and bounded by 1. Indeed since the state space
  $E$ is Polish, we may find a measurable $g$ such that
  $\dE_\mu[G|X_{s+r}]=g(X_{s+r})$ (disintegration of measure). But
  \[
  \dE_\mu[F\,P_r g(X_s)] %
  = \dE^*_\mu[F(X_{s-.}) \, P_r g(X_0)]%
  =\dE^*_\mu[f(X_0) \, P_r g(X_0)] %
  = \int\!P^*_{r/2}f\,P_{r/2}g\,d\mu
  \]
  where $f$ is similarly obtained by desintegration of the measure. Here we
  have used the notation $\dE^*_\mu$ for the expectation with respect to the
  law of the dual process at equilibrium, which is equal to the law of the
  reversed process on each interval $[0,s]$ (and conversely). We conclude
  using Cauchy-Schwarz inequality since $f$ and $g$ are still bounded by $1$.
\end{proof}

\begin{remark}\label{remduality}
  The preceding proposition implies the following comparison:
  \[
  \frac{(\al^*)^2(2t)}{\al^*(t)} \leq \al(t) \, .
  \]
  In particular if we know that $\al^*$ is ``slowly'' decreasing (i.e. there
  exists $c>0$ such that $\al^*(t) \leq c \, \al^*(2t)$), then $\al(t) \geq
  (1/c) \, \al^*(2t) \geq (1/c^2) \, \al^*(t)$. If both $\al$ and $\al^*$ are
  slowly decreasing, then they are of the same order. More generally, for
  $t\geq 2$ (for instance)
  \[
  \al^2(t) \, \leq \, \al(t/2) \, \al^*(t/2) \, \leq \, c \, \al(1) \, \al^*(t)
  \] 
  so that $\al(t) \leq c_1 \, (\al^*(t))^{1/2}$. Plugging this new bound
  in the previous inequality we obtain 
  \[ 
  \al^2(t) \, \leq \, \al(t/2)
  \, \al^*(t/2) \, \leq \, c_1 \, (\al^*(t/2))^{3/2} \, \leq \, c_1 \,
  c^{3/2} \, (\al^*(t))^{3/2}
  \]
  i.e. $\al(t) \leq c_2 \, (\al^*(t))^{3/4}$. By induction, for all $\ep >0$
  there exists a constant $c_\ep$ such that
  \[
  \al(t) \, \leq \, c_\ep \, (\al^*(t))^{1-\ep}.
  \]
\end{remark}

Again we shall mainly use the recent survey \cite{MPU} in order to compare and
extend the results of the previous section. Notice that $f \in \dL^p(\mu)$
implies that $Y \in \dL^p$.

The first main result is due to Dedecker and Rio \cite{DR,MPU}: if
$\int_0^t\!f\,P_sf\,ds$ converges in $\dL^1(\mu)$ then \eqref{eq:FCLT} holds
true under $\dP_\mu$ with $s_t^2=\Var_\mu(S_t)$ and
\[
s^2:=\lim_{t\to\infty}\frac{1}{t}s_t^2 %
=2 \int\!\PAR{\int_0^{+\infty}\!f\,P_tf\,dt}\,d\mu.
\]
In the reversible case this assumption is similar to $f\in \dD(L^{-1/2})$ (see
Remark \ref{rembak}). Using some covariance estimates due to Rio, one gets
(\cite{MPU} page 16 (37)) the following.

\begin{proposition}[FCLT via mixing]\label{pr:mix1}
  If $0\neq f\in\dL^p(\mu)$ for some $p>2$ with $\int\!f\,d\mu=0$ and
  $\int_1^{+\infty}\!t^{2/(p-2)}\,\al(t)\,\al^*(t)\,dt <\infty$, then
  \eqref{eq:FCLT} holds true under $\dP_\mu$ with
  \[
  \frac{1}{t}s_t^2=2\int\!\PAR{\int_0^{\infty}\!f\,P_tf\,dt}\,d\mu.
  \]
\end{proposition}

We shall compare all these results with the one obtained in the previous
section later, in particular by giving some explicit comparison results
between $\al$ and $\al_{p,q}$ introduced in definition \ref{df:rate}. But we
shall below give some others nice consequences of mixing.

\subsection{Self normalization with the variance and uniform integrability}
\label{secslow}

The following characterization of the CLT goes back at least to \cite{Denker}.
The FCLT seems to be less understood \cite{MPU,MP06}.

\begin{theorem}[CLT]\label{th:denker}
  Assume that $\al(t)$ (or $\al^*(t)$) goes to 0 as $t \to +\infty$ (i.e. the
  process is ``strongly'' mixing). Then for all $0\neq f\in \dL^2(\mu)$ such
  that $\int f d\mu =0$ and $\lim_{t\to\infty}\Var_\mu(S_t(f))=\infty$, the
  following two conditions are equivalent:
  \begin{enumerate}
  \item[(1)] $\PAR{\frac{S^2_t}{\Var(S_t(t))}}_{t\geq 1}$ is uniformly
    integrable
  \item[(2)] $\PAR{\frac{S_t}{\sqrt{\Var(S_t(t))}}}_{t\geq1}$ converges in
    distribution to a standard Gaussian law as $t\to\infty$.
\end{enumerate}
\end{theorem}

Note that if the process is not reversible, the asymptotic behavior of
$\int_0^s\!\PAR{\int\!f\,P_uf\,d\mu}\,du$ in unknown in general, and thus
$\Var_\mu(S_t)$ is possibly bounded.

We turn to the main goal of this section. Our aim is to show how to use the
general martingale approximation strategy (as in section \ref{subsecKV}) in
order to get sufficient conditions for $S_t^2/\Var_\mu(S_t)$ to be uniformly
integrable. To this end let us introduce some notation.
\begin{equation}\label{eq:otabis}
  \be(s) = \int\!P_sf\,P^*_sf\,d\mu 
  \quad\text{and}\quad
  \eta(t) = \int_0^t\!\be(s)\,ds
\end{equation}
\begin{equation}\label{eq:ota}
  \Var_\mu(S_t) %
  = 4\int_0^{t/2}\!(t-2s)\,\be(s)\,ds %
  = th(t).
\end{equation}
If the (possibly infinite) limit exists we denote $\lim_{t \to +\infty} h(t) =
2 V \leq +\infty$.

\begin{assumption}\label{Hpos}
  We shall say that (Hpos) is satisfied if $\be(s)\geq 0$ for all $s$ large
  enough.
\end{assumption}
Assumption (Hpos) is satisfied is the reversible case, in the non reversible
case we only know that $\int_0^t \, \eta(s) \, ds > 0$. Notice that if (Hpos)
is satisfied
\begin{equation}\label{eqvarminor}
  2t\int_0^{t/4}\!\be(s)\,ds %
  \leq \Var_\mu(S_t) %
  \leq 4t\int_0^{t/2}\!\be(s)\,ds %
  + O_{t\to\infty}(1),
\end{equation}
for $t$ large enough similarly to the reversible case, so that 
\[
2 \, \eta(t/4) \leq h(t) \leq 4 \, \eta(t/2) + O_{t\to\infty}(1).
\] 

Denker's theorem \ref{th:denker} allows us to obtain new results, at least
CLTs, using the natural symmetrization of the generator and the
forward-backward martingale decomposition. \smallskip

To this end consider the symmetrized generator $L^S=\frac 12 \, (L+L^*)$. We
shall assume that the closure of $L^S$ (again denoted by $L^S$) is the
infinitesimal generator of a $\mu$-stationary Markov semigroup $P_.^S$, which
in addition is ergodic. This will be the case in many concrete situations (see
e.g \cite{wufb}). It is then known that the Dirichlet form associated to $L^S$
is again $\cE(f,g)=\int \, \Ga(f,g) \, d\mu$. We use systematically
the superscript $^S$ for all concerned with this symmetrization.

According to Corollary \ref{co:facile} (2), we know that for a centered $f\in
\dL^2(\mu)$ there exists a $\dL^2(\mu)$ solution of the Poisson equation $L^Sg
=f$ if and only if
\begin{equation}\label{eqpoisym}
  \int_0^{+\infty} \, t \, \NRM{P_t^Sf}^2_{\dL^2(\mu)} \, dt \, < \, +\infty \, .
\end{equation}
According to remark \ref{remarkdual} we thus have 
\[
\int_0^t \, f(X_s) \, ds = - \left(M_t + (M^*)_t\right) \, ,
\]
for a forward (resp. backward) martingale $M_t$ (resp. $(M^*)_t$). In
order to use Denker's theorem, it is enough to get sufficient conditions for
both $(M_t)^2/\Var_\mu(S_t)$ and $((M^*)_t)^2/\Var_\mu(S_t)$ to be
uniformly integrable.

To this end recall first that uniform integrability of a family $F_t$ is
equivalent (La Vall\'ee-Poussin theorem) to the existence of a non-decreasing
convex function $\ga$ such that $\lim_{u \to +\infty} \, \ga(u)/u =
+\infty$ and \[ \sup_{t} \, \dE_\mu\left(\ga(F_t)\right) < +\infty \, .\]

Recall now the following strong version of Burkholder-Davis-Gundy inequalities
(see \cite{DM2}, chap. VII, Theorem 92 p.304)
\begin{proposition}\label{BDG}
  Let $\ga$ be a $C^1$ convex function such that $p:=\sup_{u>0} \, \frac{u
    \, \ga'(u)}{\ga(u)}$ is finite (i.e. $\ga$ is moderate). For any
  continuous $\dL^2$ martingale $N_.$ define $N_t^*= \sup_{s\leq t} |N_s|$.
  Then the following inequalities hold
  \[
  \frac{1}{4p} \, \NRM{N^*_t}_\ga %
  \leq \NRM{\DP{N}_t^{\frac{1}{2}}}_\ga %
  \leq 6p \, \NRM{N^*_t}_\ga \, ,
  \]
  where $\NRM{A}_\ga = \inf\{\la>0\,,\,\dE\SBRA{\ga(|A|/\la)}\leq 1\}$
  denotes the Orlicz gauge norm.
\end{proposition}

In addition Doob's inequality tells us that the Orlicz norms of $N_t^*$ and
$N_t$ are equivalent (with constants independent of $t$).

Since the brackets of the forward and the backward martingales are the same,
we are reduced to show that $ \int_0^t \, \Ga(g)(X_s) \,
ds/\Var_\mu(S_t)$ is a $\dP_\mu$ uniformly integrable family. But according to
the ergodic theorem
\begin{equation}\label{eqsymgamma}
  \frac 1t \, \int_0^t \, \Ga(g)(X_s) \, ds \textrm{ converges as $t \to +\infty$ to } \int
  \Ga(g) d\mu \, \textrm{ in $\dL^1(\dP_\mu)$.}
\end{equation}

It follows first that $\Var_\mu(S_t) = O(t)$. Otherwise
$(M_t)^2/\Var_\mu(S_t)$ would converge to $0$ in $\dL^1(\dP_\mu)$ (the same
for the backward martingale), implying the same convergence for
$S_t^2/\Var_\mu(S_t)$ whose $\dL^1$ norm is equal to $1$, hence a
contradiction. If (Hpos) is satisfied, according to \eqref{eqvarminor} we thus
have that $\eta(t)=O(1)$ (and accordingly $h(t)=O(1)$), hence
$(M_t)^2/\Var_\mu(S_t)$ and $((M^*)_t)^2/\Var_\mu(S_t)$ are uniformly
integrable. But we do not really need (Hpos) here, only a lower bound $\liminf
\Var_\mu(S_t)/t \geq c >0$. Summarizing all this we have shown

\begin{proposition}\label{pr:sym1}
  Assume that the process is strongly mixing and that \eqref{eqpoisym} is
  satisfied. Assume in addition that $\liminf \Var_\mu(S_t)/t >0$. Then
  $S_t/\sqrt{\Var_\mu(S_t)}$ converges in distribution to a standard normal
  law, as $t \to +\infty$.

  Notice that in this situation one can find some positive constants $c$ and
  $d$ such that $0<c\leq \Var_\mu(S_t)/t \leq d$ for large $t$'s, and that the
  latter is ensured if (Hpos) holds.
\end{proposition}

\subsection{A non-reversible version of Kipnis-Varadhan result}
\label{subsecKVnonS}

Finally what happens if one cannot solve the symmetrized Poisson equation, but
if $f \in \dD((-L^S)^{-1/2})$, i.e. if one can apply Kipnis-Vardahan theorem to
the symmetrized process $X_.^S$ ?

Coming back to the proof of Theorem \ref{th:KV} we may introduce $g^S_T$ so
that $\na g_T^S$ converges to some $h$ in $\dL^2$ as $T$ goes to $+\infty$.

We thus have an approximate forward-backward decomposition
\begin{equation}\label{eqsymfb}
  S_t = - \, \frac 12 \, \left(M_t^T + (M^*)_t^T\right) + \, \int_0^t \, P_{T}^Sf(X_s) \, ds \, .
\end{equation}

We first look at the corresponding forward martingale $M_t^T $ whose bracket
is given by \[ \langle M^T\rangle_t = \, \int_0^t \, |\na g_T^S|^2(X_s) \,
ds \, .\] We then have for a convex function $\ga$,
\begin{eqnarray*}
  \dE_\mu\left[\ga(\langle M^T\rangle_t/t)\right] & = & \dE_\mu \left[\ga\left( \frac 1t \,
      \int_0^t \, |\na g_T^S|^2(X_s) \, ds\right)\right] \\ & \leq & \frac 1t \, \dE_\mu \left[
    \int_0^t \, \ga(|\na g_T^S|^2)(X_s) \, ds\right] \\ & \leq & \int \, \ga(|\na
  g_T^S|^2) \, d\mu \, .
\end{eqnarray*}
Since $|\na g_T^S|$ is strongly convergent in $\dL^2$, it is uniformly
integrable. So we can find a function $\ga$ as in Proposition \ref{BDG}
such that the right hand side of the previous inequality is bounded by some
$K<+\infty$ for all $T$. Hence applying Proposition \ref{BDG} we see that
$((M^T_t)^2/t))_{(T,t)}$ is uniformly integrable. The same holds for the
backward martingale.

It remains to control 
\[
A(T,t) = \dE_\mu\left[\ga\left(\frac 1t \, \left(
      \int_0^t \, P_{T}^Sf(X_s) \, ds\right)^2\right)\right] \, .
\]
But we know that $P_T^Sf$ goes to 0 in $L^2(\mu)$. So there exists some $\ga$
such that $\ga((P_{T}^S f)^2)$ is uniformly integrable. Up to a subsequence
(we already work with subsequences) we may assume that the convergence holds
true $\mu$ almost surely, applying Vitali's convergence theorem we thus have
(we may choose $\ga(0)=0$) that \[\int \, \ga\left((P^S_{T}f)^2\right) \, d\mu
\to 0 \, \textrm{ as } T \to +\infty \, .\] We thus may apply Ces\`aro's
theorem, which furnishes some non-decreasing function $T(t)$ such that $\sup_t
A(T(t),t) < +\infty$.

We may now conclude as for the proof of Proposition \ref{pr:sym1}, obtaining
the following reinforcement which is some non-reversible version of
Kipnis-Varadhan theorem (at the CLT level), since we already proved
that 
\[
\int_0^{+\infty}\!\NRM{P_t^S f}_{\dL^2(\mu)}^2\,dt  <\infty
\] 
is ensured by the condition \eqref{eqKV}:

\begin{theorem}\label{th:symKV}
  Assume that the process is strongly mixing and that \eqref{eqKV} is
  satisfied. Assume in addition that $\liminf \Var_\mu(S_t)/t \geq c >0$ (or
  equivalently that $V_->0$). Then $S_t/\sqrt{\Var_\mu(S_t)}$ converges in
  distribution to a standard normal law, as $t \to +\infty$.

  Notice that in this situation one can find some positive constants $c$ and
  $d$ such that $0<c\leq \Var_\mu(S_t)/t \leq d$ for large $t$'s, again this
  is satisfied if (Hpos) holds.
\end{theorem}

According to the discussion after Proposition \ref{pr:sym1}, the upper bound
for the rate of convergence for $\dL^p$ functions is the worse in the reversible
situation. In a sense the previous Theorem is not so surprising. But here the
condition is written for the sole function $f$, for which we cannot prove any
comparison result.

\section{Complements and examples}
\label{secexamples}

In this section we shall first discuss in a quite ``general'' framework how to
compare all the results described in the preceding two sections. This will be
done by studying the asymptotic behavior of $P_t$. Next we shall describe
explicit examples 

\subsection{Trends to equilibrium}
\label{tractable}

In order to apply corollary \ref{co:facile} we thus have to find tractable
conditions on the generator in order to control the decay of the $\dL^2$ norm
of $P_tf$. Such controls are usually obtained for all functions in a given
class. The general smallest possible class is $\dL^\infty$ so that it is
natural to introduce Definition \ref{df:rate}.

The uniform decay rate furnishes a first $p,r$-decay rate as follows
\begin{lemma}\label{lemprate}
  If $1 \leq p \leq 2$
  \[
  \al_{p,r}(t)\leq 2^{1+(p/r)} \, \al^{\frac{r-p}{r}}(t) \, ,
  \]
  while if $2\leq p$,
  \[
  \al_{p,r}(t)\leq 2^{1+(p/r)} \, \al^{\frac 2p \, \frac{r-p}{r}}(t) \, .
  \]
\end{lemma}

\begin{proof}
  The proof is adapted from \cite{CatGui3}. Pick some $K>1$ and define
  $g_K=g\wedge K \vee -K$. Since $\int g d\mu =0$, defining $m_K= \int g_K \,
  d\mu$ it holds 
  \[
  |m_K|=\ABS{\int\!g_K\,d\mu} %
  = \ABS{\int\!(g_K - g)\,d\mu} %
  \leq \int\!(|g|-K)\,\mathbf{1}_{|g|\geq K}\,d\mu %
  \leq \NRM{g}_r^r / K^{(r-1)} \, .
  \] Similarly, 
  \[
  \NRM{g-g_K}_p^p %
  \leq \int\!|g|^p\,\mathbf{1}_{|g|\geq K}\,d\mu %
  \leq \NRM{g}_r^r / K^{r-p}.
  \]
  Using the contraction property of $P_t$ in $\dL^p(\mu)$ we have
  \begin{align*}
    \NRM{P_t g}_p 
    & \leq \NRM{P_t g - P_t g_K}_p + \NRM{P_t (g_K - m_K)}_p + |m_K| \\ 
    & \leq \NRM{P_t (g_K - m_K)}_p + \NRM{g - g_K}_p + |m_K| \\ 
    & \leq \Var_\mu^{1/2}(P_t g_K) + \NRM{g}_r^{r/p} / K^{(r-p)/p} +
    \NRM{g}_r^r / K^{(r-1)} \\ 
    & \leq \Var_\mu^{1/2}(P_t g_K) + \PAR{2 / K^{(r-p)/p}},
  \end{align*}
  the latter being a consequence of $\NRM{g}_r = 1$ and $K>1$. It follows
  \[
  \NRM{P_t g}_p \leq \al(t) \, K + 2 \, K^{-(r-p)/p}.
  \]
  It remains to optimize in $K$. Actually up to a factor 2 we know that the
  optimum is attained for $\al(t) \, K = 2 \, K^{-(r-p)/p}$ i.e. for
  $K=(2/\al(t))^{p/r}$ (which is larger than one), hence the first result.

  The second one is immediate since for $p\geq 2$, $\al_{p,\infty}(t) \leq
  \al^{\frac 2p}(t)$, and we may follow the same proof without introducing the
  variance.
\end{proof}

Note that up to a factor 2 due to the proof, the result is coherent for
$r=+\infty$.

We can complete the result by the following well known consequence of the
semigroup property
\begin{lemma}\label{lemgap}
  For $r=p \geq 1$, either $\al_{p,p}(t)=1$ for all $t\geq 0$, or there
  exist positive constants $c_p$ and $C_p$ such that $\al_{p,p}(t) \leq
  C(p) \, e^{- c_p t}$.
\end{lemma}
When the second statement is in force we shall (abusively in the non-reversible
case) say that $L$ has a spectral gap. We shall discuss in the next section
conditions for the existence of a spectral gap or for the obtention of the
optimal uniform decay rate.

Of course for $f \in \dL^p$ for some $p\geq 2$ a sufficient condition for
\eqref{eqfini} to hold is
\begin{equation}\label{eqinteg}
\int_0^{+\infty} \, \al_{2,p}(t) \, dt \, < \, +\infty \, .
\end{equation}

\begin{remark}\label{reminterpol}
  Specialists in interpolation theory certainly will use Riesz-Thorin theorem
  in order to evaluate $\al_{p,r}$. Let us see what happens.

  Consider the linear operator $T_tf =P_tf - \int f \, d\mu$. As an operator
  defined in $\dL^2(\mu)$ with values in $\dL^2(\mu)$, $T_t$ is bounded with an
  operator norm equal to $1$. As on operator defined in $\dL^\infty(\mu)$ with
  values in $\dL^2(\mu)$, $T_t$ is bounded with an operator norm equal to $2 \,
  \al(t)$. Hence $T_t$ is bounded from $\dL^r(\mu)$ to $\dL^2(\mu)$ (for
  $r\geq 2$) with an operator norm smaller than or equal to $2^{2(1 - \frac
    1r)} \, \al^{\frac{r-2}{r}}(t)$, which is (up to a slightly worse
  constant) the same result as the one obtained in lemma \ref{lemprate}. The
  same holds for the pair $(1,r)$, and then for all $(p,r)$. The main
  advantage of the previous lemma is that the proof is elementary. See also \cite{CGR} for further developments on this subject. \hfill
  $\diamondsuit$
\end{remark}

In section \ref{diff} we used $\al_{2,p}$ for $p>2$. It seems that in
full generality the relation $\al_{2,p}(t) = c_p \,
\al^{\frac{p-2}{p}}(t)$ is the best possible. However it is interesting to
notice the following duality result

\begin{lemma}\label{lemdualalpha}
  For all pair $1 \leq p < r \leq +\infty$ there exists $c(p,r)$ such
  that \[\al_{p,r}(t) \leq c(p,r) \,
  \al^*_{\frac{r}{r-1},\frac{p}{p-1}}(t) \, .\]
\end{lemma}
\begin{proof}
  If $f\in \dL^r$ is such that $\int f d\mu =0$, for all $g \in
  \dL^{\frac{p}{p-1}}$, we have \[\int P_tf \, g \, d\mu = \int P_t f \,
  \left(g-\int g d\mu\right) \, d\mu = \int f \, P_t^*\left(g-\int
    gd\mu\right) \, d\mu\] hence the result.
\end{proof}
As a consequence we obtain that
\begin{lemma}\label{lemtrendameliore}
  For $1<p \leq 2$, $\al_{1,p}(t) \leq c(p) \,
  \left(\al^*(t)\right)^{\frac{2(p-1)}{p}}$.
\end{lemma}
This result is of course much better (up to a square) than the one obtained in
lemma \ref{lemprate} in this situation, since we know that for slowly
decreasing $\al$ and $\al^*$ these functions are equivalent (up to some
constants). It can also be compared with similar results obtained in
\cite{CatGui3}.

\begin{remark}
  These results allow us to compare conditions obtained in Proposition
  \ref{pr:poissonlq}, Proposition \ref{pr:poissonfaible} on one hand, and
  Theorem \ref{th:invMPU} or Proposition \ref{pr:mix1} on the other hand.

  For example, if we use the bound obtained in lemma \ref{lemprate},
  proposition \ref{pr:poissonfaible} tells that convergence to a brownian
  motion holds provided
  \[
  \int_0^{+\infty} \, (\al(t) \, \al^*(t))^{\frac{p-2}{p}} \, dt \, <
  +\infty \, .
  \] 
  (Remark that it is exactly the condition in \cite{Jones} Theorem 5). Notice
  that as soon as $\al(t) \, \al^*(t) \, < \, 1/t$ this bound is worse
  than the one in proposition \ref{pr:mix1}, so that the mixing approach
  seems to be at least as interesting as the usual one.

  However, in the diffusion case we shall obtain in proposition
  \ref{pr:vitesse} below a better bound for $\al^*_{2,p}$. Combined with
  remark \ref{remduality}, it yields (under the appropriate hypotheses) the
  condition \[\int_0^{+\infty} \, (\al^*(t))^{\ep
    +\frac{2(p-2)}{p-1}} \, dt \, < +\infty \, ,\] for some $\ep \geq
  0$ ($0$ is allowed in the slowly decreasing case), which is better than the
  mixing condition in proposition \ref{pr:mix} as long as $\al^*(t) >
  (1/t)^{(\frac{p-1}{p-2})-\eta}$ for some $\eta \geq 0$. \hfill
  $\diamondsuit$
\end{remark}

The question is: how to find $\al$ ?

\subsection{Rate of convergence for diffusions}
\label{subsecratego}

In ``non degenerate'' situations, $\al$ is given by weak Poincar\'e
inequalities:
\begin{definition}\label{df:wp}
  $\mu$ satisfies a weak Poincar\'e inequality (WPI) for $\Ga$ with rate $\be$
  if for all $s>0$ and all $f$ in the domain of $\Ga$ (or some core) the
  following holds, 
  \[
  \Var_\mu(f) \, \leq \, \be(s) \, \cE(f,f) + s \Osc^2(f) \,
  \] 
  where $\Osc(f)= esssup f - essinf f$ is the oscillation of $f$.
\end{definition}

\begin{proposition}\label{pr:wp}\textbf{(\cite{r-w} Theorem 2.1 and Theorem
    2.3)} If $\mu$ satisfies (WPI) with rate $\be$ then both $\al(t)$ and
  $\al^*(t)$ are less than $2 \, \xi^{\frac 12}(t)$ where $\xi(t)=\inf
  \{s>0, \be(s) \, \log(1/s) \leq t \}$.

  If $L$ is $\mu$-reversible (or more generally normal) some converse holds,
  i.e. decay with uniform decay rate $\al$ implies some corresponding
  (WPI).
\end{proposition}
It is actually quite hard to check, in the reversible case, whether starting
with some (WPI) one obtains a $\xi$ which in return furnishes the \emph{same}
(WPI) (see the quite intricate expression of $\be$ in \cite{r-w} Theorem
2.3). It seems that in general one can loose some slowly varying term (like a
$\log$ for instance).

Notice that (WPI) implies the following: $\cE(f,f) =0 \, \Rightarrow \,
f \textrm{ constant}$ i.e. the Dirichlet form is non degenerate. In the
degenerate case of course, the uniform decay rate cannot be controlled via a
functional inequality. The most studied situation being the diffusion case we
now focus on it.

First we recall the following explicit control proved in \cite{BCG} Theorem
2.1 (using the main result of \cite{DFG})

\begin{proposition}\label{pr:vitesse}
  Let $L$ be given by \eqref{eqgenediff}. Assume that there exists a
  $\vphi$-Lyapunov function $V$ (belonging to the domain $\dD(L)$) for some
  smooth increasing concave function $\vphi$ and for $C$ some compact subset.
  Define $H_\vphi(t)=\int_1^t (1/\vphi(s)) ds$ and assume that $\int V d\mu <
  +\infty$.

  Then, if $\lim_{u \to +\infty}\vphi'(u)=0$,
  \[
  (\al^*)^2(t) %
  \leq C \, \PAR{\int\!V\,d\mu} \, \frac{1}{\vphi\circ H^{-1}_\vphi(t)}.
  \]
  If for $p>2$ and $q$ its conjugate, $V \in \dL^q(\mu)$ then
  \[
  \al^*_{2,p}(t) \leq C(p,\NRM{V}_q) (\al^*)^{\frac{p-2}{p-1}}(t).
  \]
  If $\vphi$ is linear, $\al^*(t)$ and $\al(t)$ are decaying like $e^{-
    \la t}$ for some $\la >0$ (see \cite{DMT,BCG,BBCG}).
\end{proposition}
Note that the latter bound is better than the general one obtained in lemma
\ref{lemprate}. Of course we may use either remark \ref{remarkdual}
(telling that we may use $\al^*$ instead of $\al$) or Remark
\ref{remduality} (comparing both rates) to apply this result. \medskip

In the same spirit we shall also recall a beautiful result due to Glynn and
Meyn \cite{GM} or more precisely the version obtained in Gao-Guillin-Wu
\cite{GGW}:

We introduce the Lyapunov control condition, as in \cite{GM,GGW}
\begin{assumption}\label{asslyapgm}
  there exist a positive function $F$, a compact set $C$, a constant $b$ and a
  (smooth) function $\theta$, going to infinity at infinity such
  that \[L^*\theta \, \leq \, - \, F \, + \, b \, \mathbf{1}_C \, .\]
\end{assumption}
Then we have the following (Theorem 3.2 in \cite{GM} and its refined version
Lemma 6.2 in \cite{GGW})
\begin{theorem}\label{th:GM}
  If Assumption \ref{asslyapgm} is satisfied and $\theta^2 \in \dL^1(\mu)$,
  the Poisson equation $Lg=f$ admits a solution in $\dL^2$ , provided $|f|\leq
  F$. Hence the usual FCLT holds
\end{theorem}
The authors get the FCLT in Theorem 4.3 of \cite{GM}, but we know how to do in
this situation.

Assumption \ref{asslyapgm} is thus enough in order to ensure the existence of
a $\dL^2$ solution of the Poisson equation for $\vphi \in \cD(\dR^d)$,
so that if this assumption is satisfied we may use Lemma \ref{lempoidistrib}
(i.e. the existence of a smooth solution (but non necessarily $\dL^2(\mu)$) to
the Poisson equation for any smooth $f$). \medskip

We shall continue this section by providing several families of examples,
starting with the one-dimensional case. These examples are then extended to
$n$-dimensional reversible Langevin stochastic differential equations using
Lyapunov conditions and results of \cite{BCG,BBCG,CGGR} to recover Poincar\'e
inequalities or weak Poincar\'e inequalities through
the use of Lyapunov conditions, and so the rate $\al^*$ or $\al$.\\
We will then consider elliptic (non necessarilly reversible) examples for
which result of \cite{DFG}, recalled in Proposition \ref{pr:vitesse},
furnishes the rate $\al^*$ and then existence of the solution of Poisson
equation and CLT where the usual Kipnis-Varadhan condition cannot be used.
Comparisons with the recent results of Pardoux-Veretennikov \cite{PV1} will be made.\\
We will end with some hypoelliptic cases such as the kinetic Fokker-Planck
equation or oscillator chains for which results of \cite{DFG,BCG} still apply,
and results of \cite{PV3} are harder to consider.
It is of particular interest in PDE theory.\\
One of the main strategy to get explicit convergence controls are Lyapunov
conditions as explained before.

\subsection{Reversible case in dimension one}

\subsubsection{General criterion for weak Poincar\'e inequalities}

We recall here results of \cite{BCR2} giving necessary and sufficient
conditions for a one dimensional measure $d\mu(x)=e^{-V(x)}dx$, associated to
the one dimensional diffusion
\[dX_t=\sqrt{2}dB_t-V'(X_t)dt\]
to satisfy a weak Poincar\'e inequality.

\begin{proposition}\cite[Theorem 3]{BCR2}\label{hardywp}
  Let $m$ be a median of $\mu$, and $\be:(0,1/2)\to \dR_+$ be non
  increasing. Let $C$ be the optimal constant such that for all $f$ and
  $0<s<1/4$
  \[
  \Var_\mu(f)\le C\,\be(s)\,\int f'^2d\mu+s \,\Osc(f)^2
  \]
  then $1/4\max(b_-,b_+)\le C\le 12\max(B_+,B_-)$ where, with $m$ a median for
  $\mu$
  \begin{eqnarray*}
    b_+&=&\sup_{x>m}\frac{\mu([x,\infty[)}{\be(\mu([x,\infty[)/4)}\int_m^xe^{V}dx\\
    B+&=&\sup_{x>m}\frac{\mu([x,\infty[)}{\be(\mu([x,\infty[))}\int_m^xe^{V}dx\\
  \end{eqnarray*}
  and the corresponding ones for $b_-,B_-$ with the left hand side of the
  median.
\end{proposition}

\subsubsection{A first particular family : general Cauchy laws}
\label{subsecexample1}

Consider the diffusion process on the line
\begin{equation}\label{eqequacauchy}
  dX_t = \sqrt 2 \, dB_t \, - \, \left(\frac{\al \, x}{1+x^2} \, + \, \frac{2 \be \, x}{(e+x^2)
      \, \log(e+x^2)}\right) \, dt
\end{equation}
for some parameters $\al>1$ and $\be \geq 0$. The model is slightly more
general than the usual Cauchy laws considering $\be=0$, but the difference
allows interesting behaviors. The corresponding generator is
\[
L =
\pd^2_{x^2} \, - \, \left(\frac{\al \, x}{1+x^2} \, + \, \frac{2 \be
    \, x}{(e+x^2) \, \log(e+x^2)}\right) \, \pd_x
\]
so that $L$ is $\mu$-reversible for \[\mu(dx) =
\frac{c(\al,\be)}{(1+x^2)^{\al/2} \, \log^\be(e+x^2)} \, dx \, .\]
It is immediate that $V(x)=x^2$ satisfies
\begin{equation}\label{eqlafonction}
  LV(x) = 2 \, \frac{1 -(\al -1) x^2}{1+x^2} \, - \, \frac{4 \be x^2}{(e+x^2) \, \log(e+x^2)}
  \,
\end{equation}
hence verifies the assumption in proposition \ref{pr:rec}. So the process
defined by \eqref{eqequacauchy} does not explode (is conservative if one
prefers), and is ergodic with unique invariant measure $\mu$, which satisfies
a local Poincar\'e inequality on any interval. \medskip

The rate $\al_{2,\infty}$ is known in this situation. Indeed, according to
Proposition \ref{hardywp}, $\mu$ satisfies a weak Poincar\'e inequality
(recall definition \ref{df:wp}) with optimal rate \[\be(s)= d(\al, \be)
\, s^{-2/{(\al -1)}} \, \log^{- 2\be/(\al-1)}(1/s) \, .\] According to
Proposition \ref{pr:wp} (and its converse in the reversible case), for large
$t$,
\[
\al_{2,\infty}(t) \, \simeq \, \xi^{\frac 12}(t) \quad \textrm{ with }
\quad \xi(t) = \, \frac{1}{t^{\frac{(\al -1)}{2}}} \, \,
\log^{\frac{(\al -1)}{2} - \be}(t) \, .
\]

In the sequel we shall only consider bounded functions $f$.

If $\al>3$ or $\al=3$ and $\be>2$, $\al_{2,\infty}^2$ is
integrable, and so we may apply Kipnis-Varadhan theorem to all bounded
functions $f$. 

Interesting cases are $\al=3$ and $\be \leq 2$.

If $\be>1$, $\theta(x)=|x|$ for large $|x|$'s satisfies the assumptions in
Theorem \ref{th:GM}, and accordingly the usual FCLT holds provided $|f(x)|
\leq c/|x|$ at infinity. If $\be\leq 1$ a similar result holds but this time
for $|f(x)| \leq c/|x|^{1+\ep}$ at infinity, for any $\ep>0$.

But it should be interesting to know what happens for bounded $f$'s that do
not go to 0 at infinity. 

\subsubsection{A second general family: subexponential laws}
\label{subsecexample2}

Let us consider the process on the line
\[dX_t=\sqrt{2}dB_t-\al x\,|x|^{\al-2}dt\]
for $\al<1$ with the generator
\[L=\pd^2_{x^2}-\al x\,|x|^{\al-2}\,\pd_x\]
which is $\nu_\al$ reversible where
\[
\nu_{\al}(dx)=C(\al)\,e^{-|x|^\al}dx.
\]
It is well known the process does not explode and ergodic with unique
invariant measure $\mu$. By Proposition \ref{hardywp}, one easily gets that
$\nu_\al$ satisfies a weak Poincar\'e inequality with
$\be(s)=k_\al\log(2/s)^{\frac{2}{\al}-2}$. According to Proposition
\ref{pr:wp} (and its converse in the reversible case), for large $t$,
\[
\al_{2,\infty}(t) \, \simeq  \, \xi^{\frac 12}(t) \quad \textrm{ with } \quad \xi(t) =
 \,  e^{-ct^{\al}}\, .
\]
It is then of course immediate by Kipnis-Varadhan theorem, and Proposition
\ref{pr:poissonlq} for tractable conditions, to get that as soon as
$f\in\dL^p$ for $p>2$ then it satisfies the FLCT. Of course, the interesting
examples are in unbounded test functions like $f(x)=e^{\frac12|x|^\al}g(x)-c$
for $g$ in $\dL^2(dx)$ but not in any $\dL^p(dx)$ for any $p>2$. We believe
that in this context, one may exhibit anomalous speed in the FCLT, as in the
Cauchy case explored in the following sections. It does not seem that
interesting new examples may be sorted out using Glynn-Meyn's result. \medskip

\subsection{Reversible case in general}

We quickly give here multidimensional Langevin-Kolmogorov reversible diffusions
example (say in $\dR^n$), that may be treated as in the one-dimensional case
using the appropriate Lyapunov conditions and weak Poincar\'e inequalities.

\subsubsection{Cauchy type measures}

Let us consider with $\al>n$
\[\mu_\al(dx):=Z\,(1+|x|^2)^{\al/2}\,dx\]
associated to the generator
\[L=\De-\frac{\al x}{1+|x|^2}.\na\]
reversible with respect to $\mu$. In fact one may use as in the one
dimensional case Lyapunov functions $W(x)=|x|^k$ for large $|x|$ so that for
large $|x|$
\[
LW=(nk+k(k-2))\,|x|^{k-2}-k\al\frac{|x|^{k}}{1+|x|^2}
\]
so that to get a Lyapunov condition we have to impose the compatibility condition $\al>n+k-2$.\\
Use now Theorems 2.8 and 5.1 in \cite{CGGR} to get a weak Poincar\'e
inequality with $\be(s)=c(n,\al) s^{-\frac{2}{\al-n}}$ leading to
\[
\al_{2,\infty}(t) %
=c'(\al,n)\frac{\log^{\frac{\al-n}{2}(t)}}{t^{\frac{\al-n}{2}}}.
\]
We then get that if $\al>n+2$ then $\al^2_{2,\infty}$ is integrable and thus
Kipnis-Varadhan theorem may be used for all bounded functions. Note that in
this case, one does not recover the
optimal speed of decay via the results of \cite{DFG}.\\
We may also use Theorem \ref{th:GM} to consider unbounded function: for $k\ge
2$, if $\al>n+2k$ and $\al>n+k-2$
then the usual FCLT holds for all centered function $f$ such that $|f|\le c(1+|x|^{k-2})$.\\
One may also, in the setting where $K\ge 2$, $f$ is centered with $|f|\le
c(1+|x|^{k-2})$ and $\al>n+2(k-2)$ (so that $f\in \dL^\be$ for
$\be<\frac{\al-n}{k-2}$), use Prop. \ref{pr:poissonlq}: if $\al>n+2k-3$
then the FCLT holds. Note that it gives better results than Theorem \ref{th:GM}.\\
One may of course generalize the model ($\be\not=0$) as in the one-dimensional
case, which would lead to the same discussion as in the one-dimensional case.

\subsubsection{Subexponential measures} 

Let us consider for $0<\al<1$,
\[
\nu_{\al}(dx)=C(\al)\,e^{-|x|^\al}dx
\]
associated to the $\nu_\al$-reversible generator
\[
L=\De-\al x\,|x|^{\al-2}\,.\na.
\]
With $W(x)=e^{a|x|^\al}$ for large $|x|$, one easily gets that for large $|x|$
\[LW(x)\le -c\al^2a(a-1)\,|x|^{2\al-2} e^{a|x|^\al}\]
so that by Theorems 2.8 and 5.1 in \cite{CGGR}, we get that $\nu_{\al}$
verifies a weak Poincar\'e inequality with
$\be(s)=k_{n,\al}\log(2/s)^{\frac{2}{\al}-2}$. We may then mimic the
results given in the one dimensional case. \medskip

\subsection{Beyond reversible diffusions}

We will focus here on general diffusion models on $\dR^n$, with the notations
of \cite{PV1,PV3} for easier comparisons,
\[
dX_t=\sig(X_t)dB_t+b(X_t)dt
\]
with generator
\[
L=\sum_{i,j=1}^na_{ij}(x)\pd^2_{x_i,x_j}+\sum_{i=1}^nb_i(x)\pd_{x_i},
\]
and $a=\sig\sig^*/2.$ We will suppose that $\sig$ is bounded and
$b,\sig$ locally (bounded) Lipschitz functions. We assume moreover a
condition on the diffusion matrix
\[
(H_\sig):\qquad  \left\langle a(x)\frac x{|x|},\frac x{|x|}\right\rangle\le
\la_+\,,\qquad Tr(\sig\sig^*)/n\le \La .
\]
Note that Pardoux and Veretennikov also impose an ellipticity condition
in \cite{PV1}, or a local Doeblin condition in \cite{PV3} preventing however
too degenerate models like kinetic Fokker-Planck ones. We also introduce the
following family of recurrence conditions
\[
(H_b(r,\al)):\qquad\forall |x|\ge M,\qquad \left\langle b(x),\frac
  x{|x|}\right\rangle\le -r|x|^\al.
\]
We suppose $M>0$, $\al\ge -1$, and when $\al=-1$, that the process does
not explode (it will be a consequence of the Lyapunov conditions given later).
We also define when $\al=-1$, $r_0=(r-\La n)/2)/\la_+$. We may then
use the results of \cite{DMT,DFG} and \cite{PV1} to get that

\[
\al_*(t)^2\le\left\{\begin{array}{ll}
\displaystyle C\,e^{-ct}&\mbox{ if }\al\ge 0,\\
\displaystyle C \, e^{-ct^{\frac{1+\al}{1-\al}}}&\mbox{ if }-1\le \al< 0,\\
\displaystyle C \, (1+t)^{-k}&\mbox{ if } \al=-1 \mbox{ and } 0<k<r_0,
\end{array}\right.
\]
for some (usually non explicit) constants $C,c>0$. Note that these results are obtained
using Lyapunov functions $W_1(x)=e^{a|x|}$, $W_2(x)=e^{a|x|^{1+\al}}$ and
$W_3(x)=1+|x|^{2k+2}$ respectively, for some
$a<\frac{2r}{\la_+(1+\al)}$ whenever $\al>-1$). Namely outside a
large ball, for some positive $\la$
\[
\al\ge0,\qquad LW_1\le -\la W_1,
\]
\[
-1<\al<0,\qquad L W_2\le -\la\,W_2 \,[\ln
W_2]^{2\frac{\al}{1+\al}},
\]
\[
\al=-1,\qquad LW_3\le -\la W_3^{\frac{m-2}m}.
\]
All this shows that the process is positive recurrent. We denote by $\mu$ its
invariant probability measure. Remark that the convergence rate in the last
case is slightly better than the one in Pardoux-Veretennikov. Note that a
direct consequence of these Lyapunov conditions is that $W_1$ is $\dL^1(\mu)$,
$W_2 \,[\ln W_2]^{2\frac{\al}{1+\al}}\in \dL^1(\mu)$ and
$W_3^{\frac{k}{k+1}}\in\dL^1(\mu)$. These last two integrability results are
presumably not optimal, indeed results of \cite[Proposition 1]{PV1} give us in
the case $\al=-1$ that for every $m<2r_0-1$, $W_4(x)=1+|x|^m$ is in
$\dL^1$.

We may then use results of Proposition \ref{pr:poissonlq}, or more precisely
Proposition \ref{pr:poissonfaible} to get results on the solution of the
Poisson equation and the FCLT that we may compare with \cite[Theorem 1]{PV1}.
Comparison is not so easy as Pardoux-Veretennikov's results consider function
$f$ with polynomial growth and obtain polynomial control of the solution of
the Poisson equation, when our results deal with $\dL^p$ control. Glynn-Meyn's
result will help us in this direction. We will only consider here examples for
$\al=-1$ and $-1<\al<0$, i.e. sub-exponential cases.

{\it Case $\al=-1$}. Pardoux-Veretennikov's result, assuming some
ellipticity condition (namely the existence of a $\la_->0$ for the
corresponding lower bound in $(H_\sig)$) establishes that if $|f(x)|\le
c(1+|x|^\be)$ for $\be<2r_0-3$ then the solution of the Poisson equation
$g$ exists with a polynomial control in $|x|^{\be+2+\ep}$ ($\ep>0 $
arbitrary) just ensuring that $g\in \dL^1$. They also obtain a polynomial
upper control of $|\na g|$.
We have not pushed too much further in this last direction but elements of the next sections may give integrability results for $|\na g|$.\\
To use Proposition \ref{pr:poissonfaible} in our context, one has to verify,
for smooth $f$ in $\dL^p$ for simplicity, that $\al(t)\al^*(t)$ is
sufficiently decreasing. Using Remark \ref{remduality}, one gets here that for
all $k<r_0$
\[
\al(t)\al^*(t)\le c_k t^{-k}
\]
and we have thus to impose the condition that $k(p-2)>p$. Our results are then
weaker than Pardoux-Veretennikov as it enables us only to consider $f$ to be
in $\dL^p$ for $p>2$ whereas they consider $f$ in $\dL^m$ for
$m<(2r_0-1)/(2r_0-3)$.

Note however that we have no ellipticity assumption, and we refer to examples
in the next paragraph, which cannot be obtained
using the results of Pardoux-Veretennikov.\\
Remark finally that our results do not only apply to the existence of the
solution of the Poisson equation but also to the FCLT, with a finite variance,
which is not at all ensured by Pardoux-Veretennikov's results. In this
perspective, if we want to use Pardoux-Veretennikov result to get a finite
variance, we will have to impose that there exists $p\ge 1$ such that
$\max(p\be, \frac{p}{p-1}(\be+2))<2r_0-1$, which will imply that for $p\ge 2$
one has to impose $(r_0-1/2)(p-2)>p$ which is slightly stronger than our
conditions. \medskip

{\it Case $-1<\al<0$}. In fact, by the results of Pardoux-Veretennikov, one
has that for $f$ bounded by a polynomial,
then $g$ is also bounded by a polynomial and thus at least in $\dL^1$.\\
We get much more general results here as we allow, for example, smooth $f$
such that there exists $C>0$ with
\[
|f(x)|\le C\,
e^{\left(\frac{r}{\la_+(1+\al)}-\ep\right)|x|^{1+\al}}
\]
for $\ep>0$.\\

Note also that no additional ellipticity condition is supposed, and even in
the subsequent work \cite{PV3}, the local Doeblin condition and condition
$(A_T)$ (see \cite[Page 1113]{PV3} seems to be verified in only slightly
degenerate case. We will then give here particular examples that may be
reached through our work. \medskip

\subsection{Kinetic models}
\label{seckine}

Consider a kinetic system, where $v$ is the velocity (in $\dR^d$) and $x$ is
the position. The motion of $v$ is perturbed by a Brownian noise, i.e. we
consider the diffusion process ${(X_t,V_t)}_{t\geq0}$ with state space
$\dR^d\times\dR^d$ solution of the kinetic stochastic differential equation
\[
\begin{cases}
  dx_t & = v_t \, dt,\\
  dv_t & = H(v_t,x_t)dt + \sqrt{2}dB_t.
\end{cases}
\]
If the initial law of $(x_0,v_0)$ is $\nu$ we denote by $P(t,\nu,dx,dv)$ the
law at time $t$ of the process. A standard scaling (see e.g. \cite{DegSeb}) is
to consider
\[
P^\ep(t,\nu,dx,dv) %
= \ep^{-d} \, P\PAR{\frac{t}{\ep^2},\nu^\ep,\frac{dx}{\ep},dv}
\]
i.e. the law of the scaled process $(\ep \, x_{t/\ep^2} \, ,
\, v_{t/\ep^2})$ (also rescale the initial law), solution of
\begin{equation}\label{edp}
  \ep\pd_t P + v\cdot\na_x P - \frac{1}{\ep}\left(\De_v P + \mathrm{div}_v(H \,
    P)\right) = 0 \, .
\end{equation}
The FCLT with $v(\ep)=\sqrt \ep$, if it holds, combined with a standard
argument of propagation of chaos (see \cite{CCM} for more details) implies
that as $\ep$ goes to $0$, $P^\ep(t,dx,dv)$ converges to the product $N(t,dx)
\, M(dv)$ where $M(dv)$ is the projection of the invariant measure of the
diffusion on the velocities space and $N(t,dx)$ is the solution of the
appropriate (depending on the asymptotic variance) heat equation on the
positions space. \smallskip

Let us present more concrete examples where we can use the results of the paper just using $f(v)=v$ or
$f(x,v)=v$, as well as the possible necessity of using another scaling in
space (anomalous rate of convergence), via explicit speed of convergence obtained as previously via Lyapunov conditions. 

{\it Kinetic Fokker-Planck equation.}\\
Let us consider the following stochastic differential system
\begin{eqnarray*}
dx_t&=&v_t\,dt,\\
dv_t&=& \sqrt{2} \, dB_t-v_t\,dt-\na F(X_t)\, dt,
\end{eqnarray*}
where $(B_t)$ is a $\dR^d$-Brownian motion. The invariant (but non-reversible)
probability measure is then $\mu(dx,dv) = Z^{-1} \, e^{- \, (\frac 12 |v|^2 +
  F(x))} \, dv \, dx$.

If $F(x)$ behaves like $|x|^p$ for large $|x|$ with $0<p<1$ then one can build
a Lyapunov function $W(x,v)$ behaving at infinity as $e^{a(|v|^2+|x|^p)}$ (for
$s$ sufficiently small) and such that outside a large ball (see
\cite{DFG,BCG})
\[LW\le -\la\, W\, [\ln W]^{2\frac{p-1}{p}}.\]
We may thus apply the results explained in the previous case $-1<\al<0$.
\medskip

{\it Oscillator chains.}\\
We present here the model studied by Hairer-Mattingly \cite{HM}: 3-oscillator
chains
\begin{eqnarray*}
  dq_0&=&p_0\, dt\\
  dp_0&=&-\ga_0p_0\, dt-q_0|q_0|^{2k-2}\, dt-(q_0-q_1)\, dt+\sqrt{2\ga_0T_0}dB^0_t\\
  dq_1&=&p_1\, dt\\
  d p_1&=&-q_1|q_1|^{2k-2}-(2q_1-q_0-q_2)dt\\
  dq_2&=&p_2\,dt\\
  dp_2&=&-\ga_2p_2\, dt-q_2|q_2|^{2k-2}\, dt-(q_2-q_1)\, dt+\sqrt{2\ga_2T_2}dB^2_t\\
\end{eqnarray*}
where $B^0$ and $B^2$ are two independent brownian motions. Then by Theorem
5.6 in \cite{HM}, if $k>3/2$, one can give a Lyapunov function $W$ for which
$LW\le -\la W^{r}+C$ for some $r<1$ so that we may use the results
presented before in the polynomial rate case.


\section{An example of anomalous rate of convergence}
\label{secexanomal}

In all the examples developed before, the asymptotic variance was existing. We
shall try now to investigate the possible anomalous rates of convergence, i.e.
cases where the variance of $S_t$ is super-linear. Instead of studying the
full generality, we shall first focus on a simple example, namely the one
discussed in section \ref{subsecexample1}. \medskip

We consider the generator $L$ defined in \eqref{eqlafonction} in the critical
situation $\al=3$ and $\be \leq 2$ or the supercritical one i.e
$\al<3$ (but $\al>1$). For simplicity we shall here directly introduce
the function $g$ and choose $g(x)=x^2$, so that $f=Lg$ is bounded but does not
go to $0$ at infinity (hence we cannot use Theorem \ref{th:GM}). \smallskip

Since $\na g(x)=2x$, $\na g \in \dL^2(\mu)$ if and only if $\al=3$ and
$\be>1$.

According to Remark \ref{remkvcauchy} we may thus apply Kipnis-Varadhan
result, so that from now on these cases are excluded. Remark that for this
particular case, Kipnis-Varadhan result applies for $\be>1$, while for the
general bounded case (i.e. $f$ bounded) we have to assume that $\be>2$. This
is presumably due to the non exact correspondence between (WPI) and the decay
rate $\xi$ as noticed just after Proposition \ref{pr:wp}. \smallskip

Our goal in this section will be to evaluate $\Var_\mu(S_t)$ and to see that
one can apply Denker's Theorem \ref{th:denker}, i.e. obtain a CLT with an
anomalous explicit rate.

In the sequel, $c$ will denote a universal constant that may change from place
to place.

For $K>0$ we introduce a truncation function $\psi_K$ such that,
$\mathbf{1}_{[-K,K]}\leq \psi'_K \leq \mathbf{1}_{[-K-1,K+1]}$ and all
$\psi''_K$ are bounded by $c$ ($\psi_K$ is thus an approximation of $x\wedge K
\vee -K$).

We then define $g_K=\psi_K(g)$, $f_K = L g_K$ which is still bounded by $c$
and such that
\[
|f_K - f| \, \leq \, c \, \mathbf{1}_{|x|\geq K} \, .
\]
In what follows, we shall use repeatedly the fact that, for large $K$
\begin{eqnarray*}
  \int_e^K \, x^{a} \, \log^{\be}(x) \, dx & \simeq & c(a,\be) \left(1 + K^{a+1} \,
    \log^{\be}(K)\right) \quad \textrm{ if } a\neq -1 \\ \int_e^K \, x^{-1} \, \log^{\be}(x) \, dx
  & \simeq & c(\be) \left(1 + \log^{\be+1}(K)\right) \quad \textrm{ if } \be\neq -1 \\
  \int_e^K \, x^{-1} \, \log^{-1}(x) \, dx & \simeq & c \left(1 + \log \log (K)\right) \, .
\end{eqnarray*}
These estimates follow easily by integrating by parts (integrate $x^a$ and
differentiate the $\log$).

Now we can write (we are using the notation in section \ref{secslow}, in
particular \eqref{eq:ota} and \eqref{eq:otabis}):

\begin{eqnarray}\label{eqvarmax1}
  (S_t)^2 & \leq &  2 \, (S_t- S_t^{f_K})^2 \, + \, 2 \, (S_t^{f_K})^2 \nonumber\\ & \leq & 2 \,
  (S_t- S_t^{f_K})^2 +  \, (M^{g_K}_t)^2 +  \, ((M^*)^{g_K}_t)^2 \, ,
\end{eqnarray}
or
\begin{equation}\label{eqvarmax2}
  (S_t)^2  \leq 2 \, (S_t- S_t^{f_K})^2 + 8 \, (g^2_K(X_t)+g^2_K(X_0)) + 4 \, (M^{g_K}_t)^2 \, ,
\end{equation}
and
\begin{equation}\label{eqvarmin}
(S_t)^2  \geq 4 \, (M^{g_K}_t)^2 \, - \, 2 \, (S_t- S_t^{f_K})^2 - 8 \,
(g^2_K(X_t)+g^2_K(X_0)) \, .
\end{equation}

Recall that \[2t \, \eta(t/4) \, \leq \, \Var_\mu(S_t) \, \leq \, 4t \,
\eta(t/2)\] with $\eta$ given in \eqref{eq:otabis} which is non-decreasing
since $L$ is reversible. Hence we know that $\Var_\mu(S_t)/t$ is bounded
below. This will allow us to improve on the results in section
\ref{subsecexample1}. \smallskip

Indeed for $K>K_0$ where $K_0$ is large enough,
\begin{eqnarray}\label{eqdiffK}
  \dE_\mu\left[(S_t- S_t^{f_K})^2\right] & \leq & c \, \dE_\mu\left[\int_0^t \, \int_0^s \,
    \mathbf{1}_{|X_s|\geq K} \, \mathbf{1}_{|X_u|\geq K} \, du \, ds\right] \nonumber \\ & \leq & c \,
  \dE_\mu\left[\int_0^t \, s \, \mathbf{1}_{|X_s|\geq K} \, ds\right] \nonumber \\ & \leq & c \, t^2 \,
  \mu(|x|\geq K) \, \leq \, c''(\al,\be) \, t^2 \, K^{1-\al} \, \log^{-\be}(K) \, .
\end{eqnarray}

\begin{eqnarray}\label{eqgrad1}
  \dE_\mu\left[(M^{g_K}_t)^2\right] & \leq & c \, \dE_\mu\left[\int_0^t \, X_s^2 \, \mathbf{1}_{|X_s|\leq
      K+1} \, ds\right] \nonumber \\ & \leq & c \, t \, \int_{-K-1}^{K+1} \, x^2 \, \mu(dx) \nonumber \\
  & \leq &  c(\al,\be) \, t \, (1 +\vphi(K)) \, ,
\end{eqnarray}
with $\vphi(K) = K^{3-\al} \, \log^{-\be}(K)$ if $\al \neq 3$,
$\vphi(K)= \log^{1-\be}(K)$ if $\al=3$ and $\be \neq 1$, and finally
$\vphi(K) = \log \log (K)$ if $\al=3$ and $\be = 1$ . Note that
similarly
\begin{eqnarray}\label{eqgrad2}
  \dE_\mu\left[(M^{g_K}_t)^2\right] & \geq &  \dE_\mu\left[\int_0^t \, X_s^2 \, \mathbf{1}_{|X_s|\leq
      K} \, ds\right] \nonumber \\ & \geq & c \, t \, \int_{-K}^{K} \, x^2 \, \mu(dx) \nonumber \\
  & \geq &  c'(\al,\be) \, t \, \left(1 \, + \vphi(K)\right) \, .
\end{eqnarray}

In addition
\begin{eqnarray}\label{eq:ongrad}
  \int \, g_K^2 \, d\mu & \leq & c \, \int_{-K-1}^{K+1} \, \frac{x^4}{(1+|x|^\al) \,
    \log^{\be}(e+|x|^2)} \, dx + 2 \, K^4 \, \mu(|x|>K) \nonumber \\ & \leq & c \, (1 +K^{5-\al}
  \, \log^{-\be}(K) ) \, .
\end{eqnarray}

According to lemma \ref{factsym} we already know that $\Var_\mu(S_t)/t$ is
bounded if and only if we are in the Kipnis-Varadhan situation (in particular
as we already saw if $\al=3$ and $\be>1$). In order to get the good order
for $\Var_\mu(S_t)/t$ by using \eqref{eqvarmax2} and \eqref{eqvarmin} we
have to choose $K(t)$ in such a way that \[\dE_\mu\left[(M^{g_K}_t)^2\right]
\gg \int \, g_K^2 \, d\mu\] and
\[\dE_\mu\left[(M^{g_K}_t)^2\right] \gg \dE_\mu\left[(S_t- S_t^{f_K})^2\right] \, .\]

Hence, according to \eqref{eqgrad1} and \eqref{eqgrad2} as well as
\eqref{eqdiffK} and \eqref{eq:ongrad} we need for $(\al,\be) \neq (3,1)$
\begin{equation}\label{eqcontrolvar}
  t \, \left(K^{3-\al} \, \mathbf{1}_{\al>3} + \log(K) \, \mathbf{1}_{\al=3}\right) \,
  \log^{-\be}(K) \, \gg \, \max(K^{5-\al} \, \log^{-\be}(K) \, ; \,  t^2 \, K^{1-\al} \,
  \log^{-\be}(K)) \, ,
\end{equation}
We immediately see that the unique favorable situation is obtained for
\begin{equation}\label{eqKt}
  \al =3 \textrm{ and } \be \neq 1 \quad \textrm{ and } \quad K^2 \, \log(K) \gg t \gg
  K^2/\log(K) \, .
\end{equation}
In this situation the leading term $\dE_\mu\left[(M^{g_K}_t)^2\right]$ is of
order $t \, \log^{1-\be}(K)$ i.e. of order $t \, \log^{1-\be}(t)$.

If $\al=3$ and $\be=1$ we get
\begin{equation}\label{eqKtbis}
  \quad K^2 \, \log(K) \, \log\log(K) \gg t \gg K^2/\log(K) \, \log\log(K)
\end{equation}
yielding this time $\dE_\mu\left[(M^{g_K}_t)^2\right] \simeq t \, \log\log(t)$.

\emph{So we now consider the cases $\al=3$ and $\be\leq 1$}.

Notice that it corresponds to the rate of convergence described in the next
section \ref{secanomalous}. \smallskip

We thus have
\begin{equation}\label{eqvarS}
\Var_\mu(S_t)/t \simeq \log^{1-\be}(t) \quad \textrm{(or $\log \log t$ if $\be=1$)} \, .
\end{equation}
Any choice of $K(t)$ satisfying \eqref{eqKt} (or \eqref{eqKtbis}) yields that
$(S_t-S_t^{f_K})^2/t \, \log^{1-\be}(t)$ (or $t \, \log \log t$) goes to
$0$ in $\dL^1(\mu)$. Hence, thanks to \eqref{eqvarmax1}, it remains to show
that $(M^{g_K}_t)^2/t \, \log^{1-\be}(t)$ (or $t \, \log \log t$) is
uniformly integrable i.e. that the bracket
\[\int_0^t \, |\na g_K|^2(X_s) \, ds /t \, \log^{1-\be}(t) \quad \textrm{ or } t \, \log \log(t)\]
is uniformly integrable, according to Proposition \ref{BDG}. Due to the form of $g_K$ it is thus
enough to show that
\begin{equation}\label{equnifK}
  H(t,X,K(t)) \, := \, \int_0^t \, X^2_s \, \mathbf{1}_{|X_s|\leq 1 + K(t)} \, ds /t \,
  \log^{1-\be}(t) \quad \left(\textrm{ or  $t \, \log \log(t)$ if $\be=1$}\right)
\end{equation}
is uniformly integrable.

\begin{remark}\label{remdecay}
  One can remark that in the situation described above, $\be(t) \ll
  \al^2(t)$, that is the decay of the $\dL^2$ norm of $P_tf$ is faster than
  the worse possible one. Indeed, as we know, $\eta(t) \sim \Var_\mu(S_t)/t
  \sim \log^{1-\be}(t)$ (or $\log \log t$ for $\be=1$) while $\al^2(t)
  \sim \log^{1-\be}(t) \, t^{-1}$ so that its primitive behaves like
  $\log^{2-\be}(t)$. \hfill $\diamondsuit$
\end{remark}

To this end, denote by $u(x,M)=|x|^2 \, \mathbf{1}_{|x|\leq 1+M}$ for $M\geq
1$, and $\bar u(x,M)=u(x,M) - \int u(.,M) \, d\mu$, and $U(t,X,M)= \int_0^t \,
u(X_s,M) \, ds$.

We know that if $\be \leq 1$, and $t>1$ for instance,
\[\Var_\mu(U(t,X,M)) = 4 \, \int_0^{t/2} \, (t-2s)
\, \left(\int \, P_s^2 (\bar u(.,M)) \, d\mu\right) \, ds \, .\]

Recall that $\al^2(s)=\al^2_{2,\infty}(s)$ is the mixing coefficient
whose expression is recalled in section \ref{subsecexample1}, i.e.
$\al^2(s)\simeq \log^{1-\be}(s) \, s^{-1}$.

A direct calculation thus yields (for $t\geq 1$)
\begin{eqnarray*}
  \Var_\mu(U(t,X,M)) & \leq & 4 \, \int_0^{t/2} \, (t-2s) \, \al^2(s) \, (1+M)^4 \, ds \\
  & \leq & 4c \, (1+M)^4 \, \int_0^{t/2} \, (t-2s)  \, \, \frac{\log^{1-\be}(1+s)}{1+s} \, ds\\
  & \leq & 4c \, (1+M)^4 \, t \, \log^{2-\be}(1+t) \, .
\end{eqnarray*}
Hence if we choose $M(t)=t^a$ with $a<1/4$, \[\Var_\mu(U(t,X,t^a))/t^2 \,
\log^{2(1-\be)t} \quad \left(\textrm{ or $(\log \log t)^2$ if
    $\be=1$}\right) \quad \to 0 \, \textrm{ as } t \to +\infty \, .\] We can
also calculate the mean \[\dE_\mu(U(t,X,t^a)) \simeq c(\be) \, t \,
\log^{1-\be}(t) \quad \left(\textrm{ or $\log \log t$ if $\be=1$}\right)\]
i.e. is asymptotically equivalent to the mean of $U(t,X,K(t))$, so
that \[\dE_\mu(U(t,X,t^a))/t \, \log^{1-\be}(t) \quad \left(\textrm{ or $\log
    \log t$ if $\be=1$}\right)\] is bounded.

It follows that $U(t,X,t^a)/t \, \log^{1-\be}(t)$ or $U(t,X,t^a)/t \, \log
\log (t)$ when $\be=1$, is uniformly integrable.

We claim that \[\left(U(t,X,K(t)) - U(t,X,t^a)\right)/t \, \log^{1-\be}(t)
\quad \left(\textrm{ or $\log \log t$ if $\be=1$}\right) \quad \to 0
\textrm{ in } \dL^1(\dP_\mu) \, ,\] so that it is uniformly integrable.
According to what precedes, it immediately follows that
$H(t,X,K(t))=U(t,X,K(t))/t \, \log^{1-\be}(t)$ (with the ad hoc
normalization if $\be=1$) is also uniformly integrable. \smallskip

It remains to prove our claim. For simplicity we choose $K(t)=t^{1/2}$ (any
allowed $K(t)$ furnishes the result but calculations are easier). Since
$U(t,X,K(t)) - U(t,X,t^a) \geq 0$ it is enough to calculate for large $t$
\[
\dE_\mu\left(U(t,X,K(t)) - U(t,X,t^a)\right) = t \, \int_{t^a}^{K(t)} \, x^2 \,
\mu(dx) \, .\] If $\be \neq 1$, the right hand side is equal
to \[\frac{1}{1-\be} \, \left(\log^{1-\be}(K(t)) -
  \log^{1-\be}(t^a)\right) \, \simeq (\log(1/2)-\log(a)) \, \log^{-\be}(t)
\, .\] If $\be=1$ it is equal to \[\log \log (K(t)) - \log \log t^a \simeq
\log(1/2) - \log(a) \, .\] Our claim immediately follows in both cases.

Let us collect the results we have obtained:
\begin{theorem}\label{th:anomalouscauchy}
  Let \[\mu_\be(dx)= p_\be(x) \, dx = c(\be) \, (1+x^2)^{-3/2}
  \log^{-\be}(e+x^2) \, dx\] be a probability measure on the line and
  $L_\be = \pd_{x^2}^2 + \na (\log p_\be) \,
\pd_x$ the associated diffusion generator for which $\mu_\be$ is reversible and ergodic.
$X^\be_.$ denotes the associated diffusion process.

For $g(x)=x^2$, $f_\be=L_\be g$ is a bounded function with $\mu$-mean
equal to 0. We consider the associated additive functional $S_t^{f_\be} =
\int_0^t \, f_\be(X_s^\be) \, ds$.

If $\be>1$ we may apply Kipnis-Varadhan result (Theorem \ref{th:KV}).

If $\be=1$, $\lim_{t \to +\infty} \Var_{\mu_\be}(S_t^{f_\be})/ t \, \log
\log t = c$ for some constant $c>0$ and we may apply Denker's theorem
\ref{th:denker}.

If $\be <1$, $\lim_{t \to +\infty} \Var_{\mu_\be}(S_t^{f_\be})/ t \,
\log^{1-\be}(t) = c$ for some constant $c>0$ and we may again apply Denker's
theorem \ref{th:denker}.
\end{theorem}

The previous theorem is really satisfactory and in a sense generic. We shall
try in the next sections to exhibit general properties yielding to an
anomalous rate of convergence.


\section{Anomalous rate of convergence. Some hints}
\label{secanomalous}

The standard strategy we used for the CLT is to reduce the problem to the use
of the ergodic theorem for the brackets of a well chosen martingale. This
requires to approximate the solution of the Poisson equation, i.e. to obtain a
decomposition of $S_t$ into some martingale terms, whose brackets may be
controlled, and remaining but negligible ``boundary'' terms. In this section
we shall address the problem of using this strategy for super-linear variance.
Hence we have to choose a correct approximation of the solution of the Poisson
equation, and to replace the ergodic theorem for the martingale brackets, by
some uniform integrability property. Again we are using the notation
\eqref{eq:otabis} and \eqref{eq:ota}. \medskip

As before, for $T>0$ depending on $t$ to be chosen later, introduce again
$g_T= - \int_0^T \, P_sf \, ds$. We thus have $Lg_T = f - P_T f$ and using
It\^o's formula
\begin{eqnarray}\label{eqpapproxb}
  S_t = \int_0^t\!f(X_s)\,ds 
  &=& g_T(X_t) - g_T(X_0) - M_t^T + \int_0^t\!P_Tf(X_s)\,ds \\
  &=& g_T(X_t) - g_T(X_0) - M_t^T + S_t^T \nonumber \\ 
  &=& - \frac 12 \, (M_t^T + (M^*)_t^T) + S_t^T, \nonumber
\end{eqnarray}
where $\DP{M^T}_t = \int_0^t\!\Ga(g_T)(X_s)\,ds$. In order to prove that
$S_t^2(f)/\Var(S_t(f))$ is uniformly integrable when $X_0\sim\mu$, we
shall find conditions for the following three propositions:
\begin{eqnarray}
  \lim_{t\to\infty}\frac{1}{\Var(S_t)}\int\!(g_T)^2\,d\mu&=&0 \label{eq001} \\
  \lim_{t\to\infty}\frac{1}{\Var(S_t)}\Var_\mu(S_t^T)&=&0 \label{eq002} \\ 
  \lim_{t\to\infty}\frac{1}{\Var(S_t)}(M_t^T)^2 &\quad& \text{is uniformly
    integrable.} \label{eq003}
\end{eqnarray}
We can replace \eqref{eq001} by
\begin{equation}\label{eq003r}
  \frac{1}{\Var(S_t)}((M^*)_t^T)^2 \quad \text{is uniformly integrable.}
\end{equation}

\subsection{Study of $\int\!(g_T)^2\,d\mu/\Var(S_t)$}
\label{subsecgt} 

We already saw that in the reversible case \[\Var_\mu(g_T) = 4 \int_0^T \, s
\, \be(s) \, ds \, \leq \, 4 T \, \eta(T) ,\] We immediately see using
\eqref{eqvarminor} that if $\frac Tt \to 0$, then $\int \, (g_T)^2 \, d\mu
/\Var(S_t) \to 0$ as $t \to +\infty$.

If $t\ll T$ then $\be$ has to decay quickly enough for
$\int\!(g_T)^2\,d\mu/\Var(S_t)$ to be bounded. The limiting case
$T=ct$ will be the more interesting in view of the second ``boundary'' term.
Note that actually we only need to study the uniform integrability of
$(g_T)^2/\Var(S_t)$, but the material we have developed do not furnish
any better result in this direction.

\subsection{Study of $\Var_\mu(S_t^T)/\Var(S_t)$}

\label{subsecstt} If $\mu$ is reversible, we have
\begin{eqnarray*}
  \Var_\mu(S_t^T) & = & 2 \, \int_0^t \, \int_0^s \, \left(\int \, P_Tf \, P_{u+T} f \, d\mu\right)
  du \, ds \\& = &  4 \, \int_0^{\frac t2} \, (t-s) \, \be(s+T)  \, ds \\ & \leq & 4t \,
  (\eta(T+(t/2)) - \eta(T)) \, ,
\end{eqnarray*}
so that, for $\Var_\mu(S_t^T)/\Var(S_t)$ to go to $0$, it is enough to
have
\[
\frac{\eta(T+\frac t2) - \eta(T)}{\eta(\frac t4)} \to 0.
\]
A similar estimate holds in the non-reversible case provided (Hpos) holds.
This time we see that the good situation is the one where $t \ll T$.

\subsection{The martingale brackets}
\label{subsecbra}

It remains to calculate the expectation of the martingale brackets $\langle
M^T\rangle_t$.
\begin{eqnarray*}
  \dE_\mu\left(\langle M^T\rangle_t\right) & = & t \, \int \, \Ga(g_T) \, d\mu \\ & = & 2t \,
  \int \left(\int_0^t \, P_sf \, (f-P_Tf) \, ds\right) \, d\mu \\ & = & 4t \, \left(2 \, \eta( T/2)
    - \eta(T)\right) \, .
\end{eqnarray*}
Hence we certainly need $\left(2 \, \eta(T/2) - \eta(T)\right)/\eta(t/4)$ to
be bounded. As for the first term this requires at least that $t$ is of the
same order as $T$. \medskip

\subsection{The good rates}
\label{subsecgoodrate} 

According to what precedes, we have to consider the case when $T$ and $t$ are
comparable. For simplicity we shall choose $T=t/2$, so that the final
condition in section \ref{subsecbra} will be automatically satisfied. The
final condition in section \ref{subsecstt} becomes
\begin{equation}\label{eqborddroit}
  \lim_{t \to +\infty} \, \,  \frac{\eta(t) - \eta(t/2)}{\eta(t/4)} \, = \, 0 \, ,
\end{equation}
while the discussion in section \ref{subsecgt} yields to
\begin{equation}\label{eqbordgauche}
  \lim_{t \to +\infty} \, \, \frac{\int_0^t \, s \, \be(s) \, ds}{t \, \int_0^{t/2} \, \be(s) \,
    ds} \, = \, 0 \, ,
\end{equation}
It is thus interesting to get a family of $\be's$ satisfying
\eqref{eqbordgauche} and \eqref{eqborddroit}. Actually since $\be$ is non
increasing,
\[
\int_{t/2}^t \, \be(s) ds \leq \, \int_0^{t/2} \, \be(s) ds\] so
that \[\int_0^{t/2} \, \be(s) ds \leq \int_0^{t} \, \be(s) ds \leq 2 \,
\int_0^{t/2} \, \be(s) ds \, .
\]
Hence, \eqref{eqbordgauche} is equivalent to
\begin{equation}\label{eqkara1}
  \lim_{t \to +\infty} \, \, \frac{\int_0^t \, s \, \be(s) \, ds}{t \, \int_0^{t} \, \be(s) \,
    ds} \, = \, 0 \, .
\end{equation}
Functions satisfying this property are known, according to Karamata's theory
(see \cite{Bing} chapter 1). Recall the definition
\begin{definition}\label{df:slow}
  A non-negative function $l$ is slowly varying if for all $u>0$, 
  \[
  \lim_{t\to+\infty}\frac{l(ut)}{l(t)} = 1.
  \]
\end{definition}
Using the direct half of Karamata's theorem (see \cite{Bing} Proposition 1.5.8
and equation (1.5.8)) for \eqref{eqkara1} to hold it is enough that
\begin{equation}\label{eqslowvary}
\be(s)=\frac{l(s)}{s} \quad \textrm{ for some slowly varying $l$.}
\end{equation}
Indeed if \eqref{eqslowvary} holds, $\int_0^t \, s \, \be(s) \, ds \sim \, t
\, l(t)$ so that \eqref{eqkara1} is equivalent to \[\lim_{t \to +\infty} \, \,
\frac{l(t)}{\int_0^{t} \, \be(s) \, ds} \, = \, 0 \, ,\] which is exactly
\cite{Bing} Proposition 1.5.9a.

The converse half of Karamata's theorem (\cite{Bing} Theorem 1.6.1) indicates
that this condition is not far to be necessary too.

Furthermore, according to \cite{Bing} Proposition 1.5.9a. if
\eqref{eqslowvary} is satisfied, then $\eta$ is slowly varying too, so that
\eqref{eqborddroit} is also satisfied. These remarks combined with the
explicit value of $\Var_\mu(S_t)$ show that the latter is then equivalent to
$4t \, \eta(t)$ at infinity.

We have obtained
\begin{proposition}\label{pr:slowvary}
  \eqref{eqbordgauche} and \eqref{eqborddroit} are both satisfied as soon as
  \eqref{eqslowvary} is. In this situation $\Var_\mu(S_t)/t$ is equivalent to
  $4 \, \eta(t)$ at infinity.
\end{proposition}

Of course if we replace \eqref{eqbordgauche} by \eqref{eq003r} we do not need
the full strength of \eqref{eqslowvary} since \eqref{eqborddroit} is satisfied
as soon as $\eta$ is slowly varying. \medskip

\subsection{Study of $(M_t^T)^2/\Var(S_t)$}
\label{subsecmartingale}

Now on we shall thus take $T=t/2$ and simply denote $M_t^T$ by $M_t$. In order
to show that $(M_t)^2/\Var(S_t)$ is uniformly integrable, we can use
Proposition \ref{BDG} yielding the following :

\begin{proposition}\label{pr:anomalousgene}
  If the process is reversible and strongly mixing and if $\eta$ given in
  \eqref{eq:otabis} is slowly varying (in particular if \eqref{eqslowvary} is
  satisfied), then there is an equivalence between
  \begin{enumerate}
  \item $\displaystyle{\frac{S_t}{2\sqrt{t\eta(t)}}}$ converges in distribution to a
    standard Gaussian law as $t \to +\infty$,
  \item $\displaystyle{\PAR{\frac{1}{t\eta(t)}\int_0^t\!\Ga(g_{t/2})(X_s)\,ds}_{t\geq 1}}$
    is uniformly integrable, where $g_{t/2}:=-\int_0^{t/2}\!P_sf\,ds$.
  \end{enumerate}
\end{proposition}

We shall say (as Denker himself said when writing his theorem) that the
previous proposition is not really tractable. Indeed in general we do not know
any explicit expression for the semigroup (hence for $g_t$). The main
interest of the previous discussion is perhaps contained in the feeling that
anomalous rate shall only occur when \eqref{eqslowvary} is satisfied.

In the next section we shall even go further in explaining: 

\subsection{Why is it delicate?}
\label{subsecdelicate}

The previous theorem reduces the problem to show that
\[
\sup_t
\dE_\mu\SBRA{\ga\PAR{\frac{1}{\Var(S_t)}\int_0^t\!\Ga(g_{t/2})(X_s)\,ds}}%
<\infty.
\] 
The first idea is to use the convexity of $\ga$, yielding
\begin{align*}
  \dE_\mu\SBRA{\ga\PAR{\frac{1}{\Var(S_t)}\int_0^t\!\Ga(g_{t/2})(X_s)\,ds}}
  &\leq\frac{1}{t}\dE_\mu\SBRA{\int_0^t\!\ga\PAR{\frac{1}{h(t)}\Ga(g_{t/2})(X_s)}\,ds}\\
  &\leq\int\!\ga\PAR{\frac{1}{h(t)}\Ga(g_{t/2})}\,d\mu
\end{align*}
so that our problem reduces to show that $\Ga(g_t)/h(2t)$ is $\mu$
uniformly integrable, or, since we assume that $\eta$ is slowly varying, that
$\Ga(g_t)/\eta(t)$ is $\mu$ uniformly integrable. 

The simplest case, namely if $\na g_t/\sqrt{h(t)}$ is strongly convergent
in $\dL^2(\mu)$, holds if and only if $\eta(t)$ has a limit at infinity, i.e.
in the Kipnis- Varadhan situation. The situation when $\eta(t)$ goes to
infinity is thus more delicate. 

It is so delicate that we shall see a natural generic obstruction. In what
follows we assume that $\eta(t) \to +\infty$ as $t \to +\infty$.

For simplicity we consider the one dimensional situation with 
\[
L
= \pd^2_{x^2} +
\pd_x(\log p) \, \pd_x
\]
$p$ being a density of probability on $\dR$ which is assumed to be smooth
($C^\infty$) and everywhere positive with $p(x) \to 0$ as $x \to \infty$.
$\mu(dx)=p(x) dx$ is thus a reversible measure, and we assume that the
underlying diffusion process is strongly mixing.

We already know that $\int \, |\pd_x g_t|^2 \, d\mu \sim 4 \, \eta(t)$.
If $|\pd_x g_t|^2/\eta(t)$ is uniformly integrable, we may find a
function $h \in \dL^1(\mu)$ such that a sequence $|\pd_x
g_{t_n}|^2/\eta(t_n)$ weakly converges to $h$ in $\dL^1(\mu)$. This implies
that $p \, |\pd_x g_{t_n}|^2/\eta(t_n)$ converges to $p \, h = \nu$ in
$\cD'(\dR)$, the set of Schwartz distributions. Notice that $\nu \in
\dL^1(\dR)$ and satisfies $\int \nu(x) dx = 4$.

Of course we may replace $f$ by $P_\ep f$ for any $\ep \geq 0$
up to an error term going to 0. Thanks to (hypo-)ellipticity we know that
$P_\ep f$ is $C^\infty$, hence we may and will assume that $f$ is
$C^\infty$, so that $g_t$ is $C^\infty$ too.

Accordingly the derivatives 
\[
\pd_x(p \, |\pd_x g_{t_n}|^2/\eta(t_n)) = \frac{p \,
  \pd_x g_{t_n}}{\eta(t_n)} \, \left(2 \, \pd^2_{x^2} g_{t_n}
  + \pd_x(\log p) \,
\pd_x g_{t_n}\right) \, \to \, \pd_x \nu \, 
\]
in $\cD'(\dR)$. But \[\pd^2_{x^2} g_{t_n}= Lg_{t_n}
- \pd_x(\log p) \, \pd_x g_{t_n} = f - P_{t_n}f - \pd_x(\log p)
\, \pd_x g_{t_n} \, ,\] so that \[\pd_x \nu = \lim \,
\frac{1}{\eta(t_n)} \, \left(2 \, p \, \pd_x g_{t_n} \, (f -P_{t_n}f) \,
  - \pd_x p \, (\pd_x g_{t_n})^2\right) \, = \, - \, \pd_x
(\log p) \, \nu \, .\] Indeed the first term in the limit goes to $0$ in
$\cD'(\dR)$ since for a smooth $\vphi$ with compact support
\begin{align*}
  \int\!\vphi\,\frac{1}{\eta(t_n)}\,2\,p\,\pd_x g_{t_n}\,(f -P_{t_n}f)\,dx
  &\leq \NRM{\vphi}_\infty %
  \frac{2}{\eta(t_n)} %
  \NRM{\pd_x g_{t_n}}_{\dL^2(\mu)} %
  \NRM{f - P_{t_n}f}_{\dL^2(\mu)} \\
  & \leq \NRM{\vphi}_\infty %
  \frac{4}{\sqrt{\eta(t_n)}} %
  \NRM{\frac{\pd_x g_{t_n}}{\sqrt{\eta(t_n)}}}_{\dL^2(\mu)} 
  \NRM{f}_{\dL^2(\mu)},
\end{align*}
and we assumed that $\eta$ goes to infinity, while for the second term we know
that $p \, |\pd_x g_{t_n}|^2/\eta(t_n)$ converges to $\nu$ and that
$\pd_x p$ is smooth.

Hence $\nu$ solves $\pd_x \nu = - \pd_x (\log p) \, \nu$ in
$\cD'(\dR)$, i.e. $\nu =c/p$ which is not in $\dL^1(\dR)$ unless $c=0$ in
which case $\int \nu \, dx \neq 4$. Accordingly $|\pd_x g_t|^2/\eta(t)$
cannot be uniformly integrable. 

Hence, contrary to all the cases we have discussed before, anomalous rate of
convergence cannot be uniquely described by the behavior of the semigroup. We
need to use pathwise properties of the process. (This sentence may look
strange since the semigroup uniquely determines the process, but the
important word here is ``path''.) 

In the situation of lemma \ref{lempoidistrib} the good strategy is to use some
cut-off of $g$ as we did in the previous section, which in a sense is generic
for this situation. 

\section{Fluctuations out of equilibrium}
\label{secout}


In this section we shall mainly discuss the CLT and FCLT out of equilibrium.
But before, we shall show that in the strong mixing case (i.e.\ uniformly
ergodic situation), the \eqref{eq:CLT} ensures the \eqref{eq:FCLT}.

\begin{proposition}[From CLT to FCLT]\label{pr:cltinv}
  Assume that the process is strongly mixing (i.e. uniformly ergodic) and that
  $\Var_\mu(S_t)=th(t)$ for some slowly varying function $h$. If
  \eqref{eq:CLT} holds under $\dP_\mu$ with $s_t^2=\Var_\mu(S_t)=th(t)$ then
  \eqref{eq:FCLT} holds with $s_t^2=\Var_\mu(S_t)=th(t)$.
\end{proposition}

\begin{proof}
  Since $h$ is slowly varying, $\Var(S_{t/\ep})\sim th(1/\ep)/\ep$ as
  $\ep\to0$. For $0\leq s < t$, define
  \[
  S(s,t,\ep) %
  = \sqrt{\frac{\ep}{h(1/\ep)}}\int_{s/\ep}^{t/\ep}\!f(X_u)\,du.
  \] 
  To prove our statement it is thus enough to show that, for indices $0\leq
  s_1 < t_1\leq s_2 < t_2 \cdots < t_N$ the joint law of
  $(S(s_i,t_i,\ep))_{1\leq i\leq N}$ converges to the law of a Gaussian vector
  with appropriate diagonal covariance matrix. Up to an easy induction
  procedure, we shall only give the details for $N=2$ and $0=s_1<t_1=s = s_2 <
  t_2=t$. For $0<s<t$ and $\la \in \dR$ define
  \[
  V(\ep,s,t,\la) = \exp \left(i \, \la \, S(s,t,f,\ep)\right) \quad ,
  \quad H(x,s,t,\ep)=\dE_x\left[V(\ep,s,t,\la)\right] \, .
  \] 
  As usual we denote by $\bar H$ the centered $H -\mu(H)$.

  We only have to show that 
  \[ 
  \lim_{\ep \to 0}\dE_\mu\SBRA{V(\ep,0,s,\la) \, V(\ep,s,t,\theta)} %
  = e^{s\, \la^2/2} \, e^{(t-s) \, \theta^2/2}.
  \]
  The main difficulty here is that $t_1=s_2=s$. We introduce an auxiliary
  time 
  \[
  s_\ep= (s/\ep) - (s/\ep^{\frac 14}).
  \] 
  We then have
  \[
  \dE_\mu \left[V(\ep,0,s,\la) \, V(\ep,s,t,\theta)\right]
  =
  \]
  \begin{eqnarray*}
    & = & \dE_\mu \left[V(\ep,0,s(1-\ep^{\frac 34}),\la) \,
      V(\ep,s(1-\ep^{\frac 34}),s,\la) \, V(\ep,s,t,\theta)\right]\\ &=&
    \dE_\mu \left[V(\ep,0,s(1-\ep^{\frac 34}),\la) \,
      V(\ep,s,t,\theta)\right] + \\ & & \, + \, \dE_\mu \left[V(\ep,0,s(1-\ep^{\frac 34}),\la) \,
      \left(V(\ep,s(1-\ep^{\frac 34}),s,\la) - 1\right) \,
      V(\ep,s,t,\theta)\right]\\ & = & A_\ep + B_\ep \, .
  \end{eqnarray*}
  Now
  \begin{eqnarray*}
    A_\ep & = & \dE_\mu \left[V(\ep,0,s(1-\ep^{\frac 34}),\la) \,
      P_{s/\ep^{\frac 34}}H(X_{s_\ep},s,t,\ep)\right] \\ & = &
    \mu(H(.,s,t,\ep)) \, \dE_\mu\left[V(\ep,0,s(1-\ep^{\frac
        34}),\la)\right] \,
    + \, \\ & & \, + \, \dE_\mu
    \left[V(\ep,0,s(1-\ep^{\frac 34}),\la) \, P_{s/\ep^{\frac 34}} \, \bar
      H(X_{s_\ep},s,t,\ep)\right]\\ & = &  \, \mu(H(.,s,t,\ep))
    \dE_\mu\left[V(\ep,0,s,\la)\right] \, + \\ & & \, + \,  \mu(H(.,s,t,\ep)) \,
    \dE_\mu\left[\left(V(\ep,0,s(1-\ep^{\frac 34}),\la)-
        V(\ep,0,s,\la)\right)\right] \, +
    \\ & & + \, \dE_\mu
    \left[V(\ep,0,s(1-\ep^{\frac 34}),\la) \, P_{s/\ep^{\frac 34}} \, \bar
      H(X_{s_\ep},s,t,\ep)\right]\\ & = & A_{1,\ep} +  A_{2,\ep} +
    A_{3,\ep} \, .
  \end{eqnarray*}
  Note that \[\lim_{\ep \to 0} \, A_{1,\ep} = e^{s \,
    \la^2/2} \, e^{(t-s) \, \theta^2/2} \, ,\] according to the CLT. For
  the two remaining terms we have
  \[
  (1/\sqrt 2) \, |A_{2,\ep}| \, \leq \, \dE_\mu
  \left[\sqrt{\frac{\ep}{h(1/\ep)}} \,
    \int_{s(1-\ep^{\frac 34})/\ep}^{s/\ep} \, |f|(X_u)
    \, du\right] \, \leq \, \sqrt{\frac{\ep}{h(1/\ep)}} \, \,
  \frac{s}{\ep^{\frac 14}} \, \, \mu(|f|) \, ,\] hence goes to $0$ as
  $\ep \to 0$. Similarly \[ | A_{3,\ep}| \, \leq \, \dE_\mu
  \left[\left|P_{s/\ep^{\frac 34}} \, \bar
      H(X_{s_\ep},s,t,\ep)\right|\right] \, = \, \int \,
  \left|P_{s/\ep^{\frac 34}} \, \bar H(.,s,t,\ep)\right| \, d\mu
  \, \leq \, \al(s/\ep^{\frac 34}) \, ,\] also goes to $0$ as
  $\ep \to 0$.
  
  In the same way \[ (1/\sqrt 2) \, |B_\ep| \, \leq \, \dE_\mu
  \left[\sqrt{\frac{\ep}{h(1/\ep)}} \,
    \int_{s(1-\ep^{\frac 34})/\ep}^{s/\ep} \, |f|(X_u)
    \, du\right] \, ,\] hence goes to $0$ as $\ep \to 0$ exactly as
  $A_{2,\ep}$. The proof is completed.
\end{proof}

\begin{corollary}\label{co:cltinv}
  If $\Var_\mu(S_t)=t \, h(t)$ for some slowly varying function $h$, we may
  replace the CLT by the FCLT in all results of section \ref{secslow} (in
  particular Theorem \ref{th:symKV}), in Theorem \ref{th:anomalouscauchy} and
  in Proposition \ref{pr:anomalousgene}.
\end{corollary}

\subsection{About the law at time $t$}


\begin{theorem}\label{th:DFG}\cite{DMT} Thm 5.2.c, and \cite{DFG} Thm 3.10 and
  Thm 3.12.

  Under the assumptions of Proposition \ref{pr:vitesse}, there exists a
  positive constant $c$ such that for all $x$,
  \[
  \NRM{P_t(x,\cdot) - \mu}_{TV} \leq c V(x) \psi(t),
  \]
  where $\NRM{\cdot}_{TV}$ is the total variation distance and $\psi$ (which
  goes to $0$ at infinity) is defined as follows: $\psi(t)=1/(\vphi \circ
  H^{-1}_\vphi) (t)$ for $H_\vphi(t)= \int_1^t \, (1/\vphi(s)) ds$, if
  $\lim_{u \to +\infty} \vphi'(u)=0$ and $\psi(t)=e^{- \la t}$ for a well
  chosen $\la >0$ if $\vphi$ is linear.

  In particular for any probability measure $\nu$ such that $V\in \dL^1(\nu)$,
  if we denote by $P_t^*\nu$ the law of the process at time $t$ starting with
  initial law $\nu$, 
  \[
  \lim_{t \to +\infty} \NRM{P_t^*\nu - \mu}_{TV} = 0.
  \]
\end{theorem}

The second result is mentioned (in the case of a brownian motion with a drift)
in \cite{CGG} and proved for a stopped diffusion in dimension one in
\cite{Cetal} Theorem 2.3. The proof given there extends immediately to the
uniformly elliptic case below thanks to the standard Gaussian estimates for
the density at time $t$ of such a diffusion, details are left to the reader

\begin{theorem}\label{th:cetal}
  In the diffusion situation \eqref{eqgenediff}, assume that the diffusion
  matrix $a$ is uniformly elliptic and bounded. Assume in addition that the
  invariant measure $\mu(dx) = e^{-W(x)} \, dx$ is reversible, and that $2
  \Ga(W,W)(x) - LW(x) \geq -c > -\infty$.

  Then for all $t>0$ and all $x$, $P_t(x,dy) = r(t,x,y) \, \mu(dy)$ with
  $r(t,x,.) \in \dL^2(\mu)$. Furthermore if $e^W \in \dL^1(\nu)$, $P_t^*\nu(dy)
  = r(t,\nu,y) \, \mu(dy)$ with $r(t,\nu,.) \in \dL^2(\mu)$.

  Consequently, if the diffusion is uniformly ergodic (or strongly mixing) and
  if $e^W \in \dL^1(\nu)$, we have again 
  \[
  \lim_{t \to +\infty} \NRM{P_t^*\nu - \mu}_{TV} = 0.
  \]
\end{theorem}

\subsection{Fluctuations out of equilibrium}
\label{subsecout}

Let $\nu$ be a given initial distribution. A direct application of the Markov property shows that
\begin{lemma}\label{lemnu}
  Assume that 
  \[
  \lim_{t \to +\infty}
  \NRM{P_t^*\nu - \mu}_{TV} = 0.
  \] 
  Let $u(\ep)>\ep$ going to $0$ as $\ep$ goes to $0$.
  For any bounded $H_1,...,H_k$, denote $H(Z_.)=\otimes H_i(Z_{t_i})$. Then
  \[
  \lim_{\ep \to 0} \left|\dE_\nu\left[H\left( v(\ep) \,
        \int_{./u(\ep)}^{./\ep} \, f(X_s) \,
        ds\right)\right]-\dE_\mu\left[H\left( v(\ep) \,
        \int_{./u(\ep)}^{./\ep} \, f(X_s) \,
        ds\right)\right]\right| = 0 \, .
  \]
\end{lemma}

As a consequence we immediately obtain

\begin{theorem}\label{th:1out}
  Let $\nu$ satisfying the assumptions of Theorem \ref{th:cetal} or Theorem
  \ref{th:DFG}. If the FCLT holds under $\dP_\mu$ (i.e.\ at equilibrium) with
  $v(\ep) \to 0$ as $\ep \to 0$ but $v(\ep) \gg \ep$, then it also holds under
  $\dP_\nu$ (i.e out of equilibrium) provided one of the following additional
  assumptions is satisfied
  \begin{itemize}
  \item $\nu$ is absolutely continuous w.r.t. $\mu$ 
  \item $\nu=\de_x$ for $\mu$ almost all $x$,
  \item $f$ is bounded.
  \end{itemize}
\end{theorem}

\begin{proof}
  Choose $u(\ep)$ such that $u(\ep) \to 0$ as $\ep \to
  0$, but with $u(\ep) \gg v(\ep)$. We may apply the previous
  lemma and to conclude it is enough to show that \[\lim_{\ep \to 0}
  v(\ep) \, \int_0^{t/u(\ep)} \, f(X_s) \, ds\] in $\dP_\nu$
  probability, which is immediate when $f$ is bounded and follows from the
  almost sure ergodic theorem in the two others cases.
\end{proof}

Several authors have tried to obtain the FCLT started from a point i.e.\ under
$\dP_x$ for all $x$, not only for $\mu$ almost all $x$, see \cite{DL1,DL3}.
Here is a result in this direction:

\begin{theorem}\label{th:2out}
  Assume that $P_t^*\nu$ is absolutely continuous with respect to $\mu$ for
  some $t>0$, that the state space $E$ is locally compact and that $f$ is
  continuous. Then if the assumptions of Theorem \ref{th:cetal} or Theorem
  \ref{th:DFG} are fulfilled, then \eqref{eq:FCLT} holds under $\dP_\nu$ as
  soon as it holds under $\dP_\mu$.
\end{theorem}

\begin{proof}
  Note that, if $P_t^*\nu$ is absolutely continuous w.r.t. $\mu$, we may apply
  the previous theorem to the additive functional $\int_t^{./\ep} \,
  f(X_s) \, ds$, i.e. we may replace $0$ by some fixed $t$. It thus remains to
  control $v(\ep) \, \int_0^t \, f(X_s) \, ds$ for the same fixed $t$.
  But since $f$ is continuous, since $X_.$ is $\dP_\nu$ almost surely
  continuous and $E$ is locally compact, $\int_0^t \, f(X_s) \, ds$ is
  $\dP_\nu$ almost surely bounded, hence goes to $0$ when $\ep \to 0$
  once multiplied by $v(\ep)$.
\end{proof}

\begin{corollary}\label{co:2out}
  If $L$ given by \eqref{eqgenediff} is elliptic or more generally
  hypoelliptic, the previous theorem applies to all initial $\nu$ satisfying
  the assumptions of Theorem \ref{th:cetal} or Theorem \ref{th:DFG}. In
  particular it applies to $\nu=\de_x$ for all $x$.
\end{corollary}

\bibliographystyle{amsalpha}
\bibliography{cattiaux-chafai-guillin_clt}

\end{document}